\documentclass[11pt]{amsart}
\usepackage{amssymb,amsmath,epsfig,graphics,mathrsfs}

\usepackage{fancyhdr}
\usepackage{hyperref}
\hypersetup{
 colorlinks   = true,
 urlcolor     = blue,
 linkcolor    = blue,
 citecolor   = red ,
 bookmarksopen=true
}

\usepackage{amssymb,amsmath,epsfig,graphics,mathrsfs}

\usepackage{graphicx}

\usepackage[usenames, dvipsnames]{color} 

\usepackage[margin=3cm]{geometry}
\usepackage{hyperref}
\hypersetup{
 colorlinks   = true,
 urlcolor     = blue,
 linkcolor    = blue,
 citecolor   = red ,
 bookmarksopen=true
}

\def \phi {\varphi}

\def \R {\mathbb{R}}

\def \vf{\varphi}
\def \S {\mathbb{S}}

\newcommand{\Rn}{\mathbb R^n}

\newcommand{\Om}{\Omega}

\newcommand{\p}{\partial}

\newcommand{\la}{\lambda}

\numberwithin{equation}{section}

\newcommand{\beq}{\begin{equation}}
\newcommand{\bea}[1]{\begin{array}{#1} }
\newcommand{\eeq}{ \end{equation}}
\newcommand{\ea}{ \end{array}}

\newcommand{\ve}{\varepsilon}

\newcommand{\B}{\mathbb{B}}

\newcommand{\var}{\operatorname{div}}

\newtheorem{theorem}{Theorem}[section]
\newtheorem{lemma}[theorem]{Lemma}

\newtheorem{corollary}[theorem]{Corollary}
\newtheorem{remark}[theorem]{Remark}
\newtheorem{definition}[theorem]{Definition}

\numberwithin{equation}{section}

\begin{document}

\title[The thin obstacle problem]{The thin obstacle problem for some variable coefficient degenerate elliptic operators}

\subjclass[2010]{}

\date{}

\begin{abstract}
In this paper, we establish the optimal interior regularity and the $C^{1,\gamma}$ smoothness of the regular part of the free boundary in the thin obstacle problem for a class of degenerate elliptic equations with variable coefficients. 
\end{abstract}


\author{Agnid Banerjee}
\address{TIFR CAM, Bangalore-560065} \email[Agnid Banerjee]{agnidban@gmail.com}
 \thanks{First author is supported in part by SERB Matrix grant MTR/2018/000267 and by Department of Atomic Energy,  Government of India, under
project no.  12-R \& D-TFR-5.01-0520   }
 
\author{Federico Buseghin}
\address{Dipartimento d'Ingegneria Civile e Ambientale (DICEA)\\ Universit\`a di Padova\\ Via Marzolo, 9 - 35131 Padova,  Italy}

\email{federico.buseghin@studenti.unipd.it}

\author{Nicola Garofalo}

\address{Dipartimento d'Ingegneria Civile e Ambientale (DICEA)\\ Universit\`a di Padova\\ Via Marzolo, 9 - 35131 Padova,  Italy}
\vskip 0.2in
\email{nicola.garofalo@unipd.it}

\thanks{Third author is supported in part by a Progetto SID (Investimento Strategico di Dipartimento) ``Non-local operators in geometry and in free boundary problems, and their connection with the applied sciences", University of Padova, 2017, and also by a Progetto SID: ``Non-local Sobolev and isoperimetric inequalities", University of Padova, 2019}

\maketitle

\tableofcontents

\section{Introduction and statement of the results}\label{S:intro}

In this paper we prove the optimal interior regularity  and the $C^{1,\gamma}$ local regularity of the regular part of the free boundary in the following degenerate thin obstacle problem with variable coefficients: 
\begin{equation}\label{La0}
\begin{cases}
\var(|y|^a A(x) \nabla_X U) = 0,\ \ \ \ \ \text{in}\ \mathbb B_1^+,
\\
\min\{U(x,0)-\psi(x),-\p_y^a
U(x,0)\}=0 & \text{on}\ B_1,
\end{cases}
\end{equation}
 where for $x\in \Rn$, $y>0$, we have indicated $X = (x,y)\in \R^{n+1}_+$, and defined 
$\partial_y^a U(x,0)\overset{\rm def}{=}\lim_{y\to 0^+} y^a\partial_y U(x,y).$
The function $\psi$ is called the \emph{thin obstacle} since it is defined in the thin set $\Rn\times\{0\}$. What makes the problem \eqref{La0} degenerate is the presence of the weight $|y|^a = \operatorname{dist}(X,\{y=0\})^a$, where the parameter $a$ is allowed to range in the interval $(-1,1)$. We recall that the coincidence set is $\Lambda_\psi(U) = \{x\in B_1\mid U(x,0) = \psi(x)\}$, and that the free boundary $\Gamma_\psi(U)$ is the topological boundary (in the relative topology of $B_1$) of the set $\Lambda_\psi(U) $. While we refer the reader to Section \ref{S:prelim} for a detailed account of notations and hypothesis, here we confine ourselves to mention that throughout the present work we assume that the matrix-valued function $A$ is uniformly elliptic with Lipschitz continuous coefficients satisfying \eqref{type} below. We emphasise that the interest in studying a problem such as \eqref{La0} with variable coefficients is not merely academic: in concrete situations the separating thin manifold is not necessarily flat, and if one flattens it one is led to a problem of the form \eqref{La0}. 

With this being said, our first main result concerning the optimal interior regularity is the following.  
 
 \begin{theorem}\label{optreg}
 Assume $0\le a<1$. Let $U$ be a solution to \eqref{La0} with $\psi \in C^{1,1}$. Then $U \in C^{\frac{3-a}2}(\overline{\B_{\frac{1}{2}}^+})$, $\nabla_x U \in C^{\frac{1-a}2}(\overline{\B_{\frac{1}{2}}^+})$  and $y^a U_y \in C^{\frac{1+a}2}(\overline{\B_{\frac{1}{2}}^+})$. 
 \end{theorem}

 To state our second main result we need to introduce the notion of
regular free boundary points. An equivalent  definition of such points based on   Almgren type 
frequency function is given in Section~\ref{fbreg} below, see Definition~\ref{D:reg}. We say that a free boundary point  $(x_0, 0)\in
\Gamma_\psi(U)$ is \emph{regular} if there exist constants $0<\alpha\le \beta<\infty$ such that
\begin{align*}
\alpha\le \limsup_{r \to 0} \frac{
             \|U-\psi\|_{L^{\infty}(B_r(x_0))}}{r^{\frac{3-a}2}} \le \beta.
\end{align*}
We denote by $\Gamma_\psi^{\frac{3-a}2}(U)$ the set of all
regular free boundary points, and call it the
\emph{regular set}. The following is our second main result.
 
 \begin{theorem}\label{smooth}
Suppose that $0\le a<1$ and let $U$ be as in Theorem \ref{optreg} above.  Then, $\Gamma_\psi^{\frac{3-a}2}(U)$ is a
relatively open subset of $\Gamma_\psi(U)$. After possibly a translation and rotation of the 
coordinate axes in the thin space $\R^n\times \{0\}$, the set $\Gamma_\psi^{\frac{3-a}2}(U)$ is locally given as a graph
$$
x_n=g(x_1,\ldots,x_{n-1}),
$$
with $g\in C^{1+\gamma} $.  
\end{theorem}

We emphasise that the interior regularity claimed in Theorem \ref{optreg} is best possible, even when $A = \mathbb I$. In this context the optimal interior regularity as well as the $C^{1}$ smoothness of the regular set were established in the pioneering work \cite{CSS} in the full range $-1<a<1$. 
We also mention that for the case $a = 0$ Theorems \ref{optreg} and \ref{smooth} were first respectively established in \cite{GG}  and  \cite{GPG}. More recently, and still for the case $a =0$, the authors of \cite{JPS} have succeeded in treating the more general case of almost minimisers and H\"older variable coefficients. 

For variable coefficients thin obstacle problems such as \eqref{La0} prior to the present paper there have been no contributions to the optimal interior regularity or the $C^{1,\gamma}$ regularity of the regular free boundary when $a\not= 0$. One of the things that makes the analysis particularly challenging is the lack of those fundamental initial results such as the H\"older continuity up to the thin set of the solution $U$, its weighted Neumann derivative $y^a \p_y U$ and that of $\nabla_x U$. As it is well-known, such results represent the main building blocks in the study of lower-dimensional obstacle problems. Once they are available, the next challenge is to develop suitable monotonicity formulas that play a critical role in the blowup analysis.     
 
The aim of the present paper is to fill this gap, at least when $0\le a<1$. The range $-1<a<0$ remains presently open, but we hope to return to this question in a future study. We mention that the limitation $a\ge 0$ in Theorems \ref{optreg} and \ref{smooth} stems from Theorem \ref{kd} below and that, with the exception of such technical result, the remainder of the work in the present paper covers the full range $-1<a<1$ without any changes. It is also worth recalling here that at a local level the thin obstacle problem \eqref{La0} is known to be equivalent to the following nonlocal obstacle problem 
\begin{equation}\label{nonloc}
\min\{(-\var(B(x) \nabla)^s u, u -\psi\} =0,\ \ \ \ \ \ \ \ \ \ 0<s<1.
\end{equation}
The connection between the parameters $a$ in problem \eqref{La0} and $s$ in \eqref{nonloc} is given by $s= \frac{1-a}{2}$ and the matrix-valued function $B(x)$ is connected to $A(x)$ by formula \eqref{type} below.
With this in mind, it is evident that Theorems \ref{optreg} and \ref{smooth} only presently cover the range $s \in (0, \frac{1}{2}]$ in \eqref{nonloc}, leaving open the remaining interval $\frac 12 < s<1$.

The paper is organised as follows. In Section \ref{S:prelim}, we introduce some basic notations and gather some preliminary results that will be subsequently needed in our work. Theorem \ref{odd1} is the main regularity result about odd solutions. We stress that it cannot be extracted from the existing works. Section \ref{initial} constitutes one of the essential new contributions of the present paper. Its central results are Theorems \ref{holdery} and \ref{alp} which provide the above mentioned regularity theorems which are necessary to develop the analysis in the reminder of our work. Our approach relies on a delicate adaptation of the method of Campanato coupled with a new quantitative regularity estimate for the constant coefficient problem studied in \cite{CSS}.
Section \ref{montfor} is devoted to proving Theorem \ref{amon}. The latter is a new Almgren type monotonicity formula for \eqref{La0} that generalises the one in \cite{GG} for the case $a = 0$ and which plays a fundamental role in the rest of the paper. In Section \ref{optregu} by combining such a monotonicity formula with the a priori estimates established in Section \ref{initial} we use a blowup analysis to establish the optimal regularity in Theorem \ref{optreg}. Finally, in  Section \ref{fbreg} we prove a Weiss type monotonicity formula which, together with the epiperimetric inequality obtained in \cite{GPG}, allows us to obtain the $C^{1, \gamma}$ regularity of the regular part of the free boundary in Theorem \ref{smooth}.
 
In closing, we mention that the theory of thin obstacle problems is by now quite well developed and has several important ramifications.  We refer the interested reader to \cite{C, AU, AC, ACS, CSS, GP, PSU, GG, GPG, FS,  KRS1, DS, GPPS, CDS, DGPT, KRS2, FS2, ACM, GRO, JP, BDGP1, JPS, BDGP2} and the references therein.


\section{Preliminaries}\label{S:prelim}

In this section we introduce the notations and gather some preliminary results which will be needed in our work.   We consider the thick space $\R^{n+1}$ with generic variable $X = (x,y)$, where $x\in \Rn$, $y\in \R$, and let $|X| = (|x|^2 + y^2)^{\frac{1}{2}}$. The thin space $\Rn\times\{0\}$ will be routinely identified with $\Rn$. We denote by $\mathbb B_r = \{X\in \R^{n+1}\mid |X|<r\}$ the ball of radius $r$ centred at the origin in the thick space, and we indicate with $\mathbb B_r^{+} = \{X\in \R^{n+1}\mid |X|<r,\ y>0\}$ its upper part. The symbol $\mathbb B_r^-$ will indicate the corresponding lower part of $\mathbb B_r$. We denote by $\S_r =\p \B_r = \{X\in \R^{n+1}\mid |X| = r\}$ the sphere of radius $r$ in the thick space, and we indicate with $\S_r^+ = \S_r\cap \{y>0\}$ its upper part.
 We use the notation $B_r = \{(x,0)\in \mathbb B_r \mid |x|<r\}$ for the unit ball in the thin space $\Rn$. We assume that $X\to A(x) = [a_{ij}(x)]$ is a given symmetric, uniformly elliptic matrix-valued function of the form
\begin{equation}\label{type}
a_{ij}(x) = \sum_{i,j=1}^n b_{ij} (x) e_i \otimes e_j  + e_{n+1} \otimes e_{n+1},\end{equation} 
where $b_{ij}$'s are Lipschitz continuous and independent of $y$. This assumption will remain in force throughout the rest of the paper and will not be repeated further. 

Given a number $a \in (-1, 1)$, and a function $\psi$ in
$B_1$, known as the \emph{thin obstacle}, we consider the
problem of finding a function $U$ in $\B_1^+\cup B_1$ such that: 
\begin{equation}\label{La}
\begin{cases}
\var_X(y^a A(x) \nabla_X U) = 0,\ \ \ \ \ \text{in}\ \mathbb B_1^+,
\\
\min\{U(x,0)-\psi(x),-\p_y^a
U(x,0)\}=0 & \text{on}\ B_1,
\end{cases}
\end{equation}
 where we have defined 
\begin{equation}\label{day}
  \partial_y^a U(x,0)\overset{\rm def}{=}\lim_{y\to 0^+} y^a\partial_y U(x,y).
\end{equation}
For notational convenience, we will hereafter write $\var$ and $\nabla$ for respectively $\var_X$ and $\nabla_X$. Also, it will be important for the reader to keep in mind that in the applications of the divergence theorem to the domain $\B^+_1$ the orientation of the outer unit normal is opposite to that used in \eqref{day}. In this respect, We explicitly note for subsequent use that if we denote by $\nu$ the outer unit normal to the boundary $\p \B_1^+ = \S_1^+ \cup B_1$ of the upper half-ball, then from \eqref{type} we have $A(x) \nu = - e_{n+1}$ in $B_1$, and consequently for a function $U$ we have in $B_1$
\begin{equation}\label{cothin}
\langle \nabla U,A(x) \nu\rangle = - \p_y^a U(x,0).
\end{equation}

For later purposes we now consider in the ball $\B_1$ the degenerate pde in \eqref{La}, but with a non-zero right-hand side of the form 
 \begin{equation}\label{Laf}
\var(|y|^a A(x) \nabla V) = |y|^a f,
\end{equation}
where $f\in L^\infty(\B_1)$. By a solution to \eqref{Laf} we mean a weak solution.
For the next result see  \cite[Theorem 1.2]{STV1}.

\begin{theorem}\label{even1}
Let $V$ be an even in $y$ solution to \eqref{Laf}. Then, $V \in C^{1, \alpha}_{loc}(\B_1)$ for any $\alpha \in (0, 1)$ and the following estimate holds
\[
||V||_{C^{1,\alpha}(\B_{\frac{1}{2}})} \leq C \left(||V||_{L^{2}(\B_1, |y|^a dX)} + ||f||_{L^\infty(\B_1)}\right),
\]
where $C>0$ depends also on $\alpha$. 
\end{theorem}

Our next result, Theorem \ref{odd1}, concerns regularity of odd solutions. In preparation for it we establish the following crucial intermediate result.

\begin{lemma}\label{inter}
Let $V$ be a  solution to \eqref{Laf} such that $V \equiv 0$ on $\B_1 \cap \{y=0\}$. Then, given $\beta < \min\{1-a, 1\}$, there exists a  $C^{\beta+a}$ function $b: B_{\frac{1}{2}} \to \R$, and a constant $C$ in the form
\[
C =\tilde C(n,a, \beta,   ||A||_{C^{0,1}}) \big(||V||_{L^{2}(\B_1, |y|^a dX)} + ||f||_{L^\infty(\B_1)}\big)^2,
\]
such that the following estimate holds for every $(x_0, 0) \in \B_{\frac{1}{2}} \cap \{y=0\}$ and $r < \frac{1}{4}$,
\begin{equation}\label{dc1}
\int_{\B_r^+((x_0,0))}  \left( V(X) -    b(x_0) y^{1-a}\right)^2 y^a dX \leq C r^{n+1+a +2(1+\beta)}.
\end{equation} 
\end{lemma}

\begin{proof}
The proof is divided into several steps. We first establish \eqref{dc1} when $(x_0, 0)= (0,0)$.   
Furthermore, by a change of coordinates we can also assume that $A(0,0) = \mathbb{I}$.  

\medskip

\noindent \emph{Step 1:} 
We begin by making the observation that when   $f \equiv 0$ and $A \equiv \mathbb{I}$ the function $g= y^a V_y$ can be  evenly extended  across $\{y=0\}$ so  that it is a solution of 
\begin{equation}\label{tyy1}
\var(|y|^{-a}  \nabla g) =0.
\end{equation}
From Theorem \ref{even1}  it follows in particular that  up to the thin set $\{y=0\}$ we have $y^a V_y \in C^{\gamma}$ for all $0<\gamma \leq1$, with bounds depending only on $\int_{\mathbb{B}_1^+} V^2 y^a  dX$.  From  the proof of \cite[Theorem 4.1 (2)]{CSt} (more precisely, by applying  \cite[Lemma 4.6]{CSt}  in the limit $k \to \infty$  with $\beta_0= \alpha -a $ where $\alpha \in (0,1)$), and using $V(x,\cdot) \equiv 0$, it follows that  given $\beta_0 < \min\{1-a, 1\}$, there exists $C$ depending also on $\beta_0$  such that the following decay estimate holds at any arbitrary point $(x_0,0) \in B_{\frac{1}{2}} \times \{y=0\}$ for all $r < \frac{1}{2}$, 
\begin{equation}\label{dc0}
\int_{\B_r^+((x_0,0))}  \left( V(X) - \frac{1}{1-a}  \p_y^a V(x_0, 0) y^{1-a}\right)^2 y^a dX \leq C r^{n+1+a +2(1+\beta_0)}.
\end{equation}

\medskip

\noindent \emph{Step 2:} We  now  make the following claim: 
given $\ve>0$, there exists $\delta >0$ such that  for any $V$  which solves \eqref{Laf}, with
\begin{equation}\label{red1}
||V||_{L^{2}(\B_1^+, |y|^a dX)} \leq 1,\ \ \ ||A(x) -\mathbb{I}||_{ C^{0,1}} \leq \delta,\ \ \ ||f||_{L^{\infty}} \leq \delta,
\end{equation}
there exists $V_0$ which solves 
\begin{equation}\label{v00}
\begin{cases}
 \operatorname{div}(y^a \nabla V_0) =0\ \text{in $\B_1^+$}
 \\
 V_0=0\ \text{on $\{y=0\}$},
 \end{cases}
 \end{equation}
   with $||V_0||_{L^{2}(\B_1^+, |y|^a dX)} \leq 1$, such that
\begin{equation}\label{red2}
\int_{\B_{\frac{1}{2}}^+}   ( V -V_0)^2  y^a dX \leq \ve.
\end{equation}
The proof of this claim follows by a standard argument by contradiction as  that of Corollary 3.3 in \cite{CSt}. 

\medskip

\noindent \emph{Step 3:}  Next, we claim that there exist universal $\delta, \lambda \in (0,1)$ such that if \eqref{red1} holds, then for some constant $b_0$  with universal bounds, for all $\beta < \min\{1-a, 1\}$ we have  
\begin{equation}\label{dec0}
\int_{\B_{\lambda}^+} ( V- b_0 y^{1-a})^2 y^a dX \leq \lambda^{n+1+q+2(1+\beta)}.
\end{equation}
To establish \eqref{dec0} we first choose some $\beta_0$ such that $\beta < \beta_0 < \min\{1-a, 1\}$. Then, we note that  for a given $\ve>0$, the estimate \eqref{red2}  holds for some $V_0$ which solves \eqref{v00},  provided  the conditions in \eqref{red1} are satisfied for an appropriate $\delta$ depending on $\ve$.  Subsequently,  given such a $\beta_0$,  we  have that the estimate \eqref{dc0}   holds for  $V_0$.   Thus  it follows  from \eqref{dc0} and \eqref{red2}  that,  for  $b_0= \frac{\p_y^a V_0(0,0)}{1-a}$, we have that  for any $\lambda < \frac{1}{2}$ the following estimate holds  
\begin{equation}\label{dc2}
\int_{\B_\lambda^+((x_0,0))}  \left( V(X) -  b_0 y^{1-a}\right)^2 y^a dX \leq C \lambda^{n+1+a +2(1+\beta_0)} + C\ve.
\end{equation}
Since $\beta_0>\beta$ we can now  choose $\lambda>0$ such that
\[
C \lambda^{n+1+a +2(1+\beta_0)}= \frac{\lambda^{n+1+a+2 (1+\beta)}}{2}.\]
Subsequently, we choose $\ve>0$  such that $C\ve=\frac{\lambda^{n+1+a+2 (1+\beta)}}{2}$ which decides the choice of $\delta$ and thus \eqref{dec0} follows. 

\medskip

\noindent \emph{Step 4:} We now show that,  under the assumptions \eqref{red1}, for $\delta, \lambda$ as in \emph{Step 3} we have that for every $k=\mathbb N$ there exists $b_k$ such that  the following holds 
\begin{equation}\label{iter}
\begin{cases}
\int_{\B_{\lambda^k}^+} ( V- b_k y^{1-a})^2 y^a dX \leq \lambda^{k(n+1+a+2(1+\beta))},
\\
|b_k - b_{k+1}| \leq C \lambda^{k(\beta+a)}.
\end{cases}
\end{equation}
We  prove \eqref{iter} by induction. We note that the case $k=1$ is proven in \emph{Step 3}. Assume then that \eqref{iter} hold up to some $k\ge 2$. We let
\[
\tilde V(X)= \frac{V (\lambda^k X)  - b_k \lambda^{k (1-a)} y^{1-a}}{\lambda^{k(1+\beta)}}.\]
Since \eqref{iter} holds for  $k$ it follows by a change of variable that
\[
|| \tilde V||_{L^{2}( \B_1^+, |y|^a dX)} \leq 1.
\]
Moreover, $\tilde V$ solves
\begin{equation}
\begin{cases}
\var (y^a A_k(x) \nabla\tilde V) = y^a f_k\ \text{in $\B_1^+$},
\\
\tilde V=0\ \text{on $\{y=0\}$},
\end{cases}
\end{equation}
where $A_k(x)= A(\lambda^k x)$ and $f_k(X)= \lambda^{k(1-\beta)} f(\lambda^{k}X)$. By $\lambda^k, \beta <1$, we see that the conditions  in \eqref{red1} are satisfied and thus applying the  conclusion of \emph{Step 3} to $\tilde V$ we infer that there exists some $\tilde b_0$ with universal bounds such that
\begin{equation}\label{dc5}
\int_{\B_\lambda^+((x_0,0))}  \left( \tilde V(X) -  \tilde b_0 y^{1-a}\right)^2 y^a dX \leq  \lambda^{n+1+a +2(1+\beta)}.
\end{equation}
By scaling back to $V$, and letting $b_{k+1}= b_k + \lambda^{k(\beta+a)} \tilde b_0$, we see that \eqref{iter} is satisfied for $k+1$. By induction we infer that \eqref{iter} holds for all $k$.

\medskip

\noindent \emph{ Step 5:} Given $V$ as in the hypothesis of the lemma,  and defining $V_{r_0}(X) = V(r_0X)$, we note that $V_{r_0}$ solves
\[
\var(y^a A(r_0 x) \nabla V_{r_0}) = y^a r_0^2 f(r_0 x).\]
Therefore,  by choosing $r_0$ small enough we can ensure that  $||A(r_0 \cdot) -\mathbb{I}||_{C^{0,1}} \leq \delta,$ where $\delta$ is as in \eqref{red1}. Subsequently, by letting 
 $$W= \frac{V_{r_0}}{||V_{r_0}||_{L^2(\B_1^+, y^a dX)} + \frac{r_0^2 ||f||_{L^{\infty}}}{\delta}},$$ we see  that $W$  solves \eqref{Laf} and that all the  assumptions  in \eqref{red1} are satisfied.  We can  thus let $W$ be our new $V$  and then  by applying the conclusion of \emph{Step 4}, we have that the estimate  \eqref{iter}  holds for $W$.  The estimate \eqref{dc1} for $W$ follows from \eqref{iter}   by a standard real analysis argument with $b(0) = \lim_{k \to \infty} b_k$. In conclusion, the estimate   \eqref{dc1} holds for $V$  at $(0,0)$. By translation we finally infer that \eqref{dc1} holds for every  $(x_0, 0) \in \B_{\frac{1}{2}} \cap \{y=0\}$. The $C^{\beta+a}$ H\"older continuity of the function $b$ also follows in a standard way.

\end{proof}

We can now prove the relevant regularity result for odd solutions that is needed in this work. 
We emphasise that in the proof of the next theorem we cannot appeal to \cite[Theorem 1.6]{STV2} because that result requires $f$ to have certain decay near $y=0$ which does not generically hold in our situation.

\begin{theorem}\label{odd1}
Let $V$ be an odd in $y$ solution to \eqref{Laf}. Then, $\nabla_x V(\cdot, y) \in C^{\beta}(\overline{\B_{\frac{1}{2}}^+})$  for all $0<\beta < \min\{1-a, 1\}$. Moreover, the following quantitative estimate holds
\begin{equation}\label{even2}
||\nabla_x V||_{C^{\beta}(\overline{\B_{\frac{1}{2}}^{+}})}    \leq  C(n,a, \beta,   ||A||_{C^{0,1}}) \left(||V||_{L^{2}(\B_1, |y|^a dX)} + ||f||_{L^\infty(\B_1)}\right).
\end{equation}
Furthermore, when $f(x, y)\equiv f(x)$ for $y >0$ (i.e., when $f$ is independent of $y$ when  $y>0$), we have that  $y^a V_y \in C^{\alpha}(\overline{\B_{\frac{1}{2}}^+})$ for  all $0< \alpha < \min\{1+a,1\}$  and the following estimate holds
\begin{equation}\label{even4}
||y^a V_y||_{C^{\alpha}(\overline{\B_{\frac{1}{2}}^{+}})} \leq  C(n,a,  \alpha,  ||A||_{C^{1}}) \left(||V||_{L^{2}(\B_1, |y|^a dX)} + ||f||_{L^\infty(\B_1)}\right). \end{equation}
\end{theorem}

\begin{proof}
We first note that, given $\beta < \min\{1-a, 1\}$, in view of Lemma \ref{inter} the estimate \eqref{dc1} holds for some  $b \in C^{\beta+a} (B_{\frac{1}{2}})$. 
Before proceeding further we also  remark that, as  previously noted in the proof of Lemma \ref{inter}, for a   given  $(x_0,0)  \in \B_{\frac{1}{2}}^+ \cap \{y=0\}$ we have that  $w \overset{def}{=}  V(X)  - b(x_0) y^{1-a}$ solves in $\{y>0\}$
\begin{equation}\label{gl1}
\var(|y|^a A(x) \nabla w) =|y|^a f.
\end{equation}
Let now $X_1=(x_1, y_1)$ and $ X_2=(x_2, y_2)$ be two points in $\B_{\frac{1}{2}}^+$.  Without loss of generality we assume that  $y_1 \leq y_2$. There are two cases:
\begin{itemize}
\item[(1)] $ |X_1-X_2| \leq  \frac{1}{4} y_1$;
\item[(2)] $|X_1 - X_2| \geq \frac{1}{4} y_1$.
\end{itemize}
If (1) occurs,  then applying \eqref{dc1} with $r= \frac{y_1}{2}$, it ensues that the following $L^{2}$ bound is satisfied by $w_1(X)\overset{def}{=}  V(X) - b(x_1) y^{1-a}$
\begin{equation}\label{dc2}
\int_{\B_{\frac{y_1}{2}}(X_1)}  w_1^2 |y|^a \leq C y_1^{n+1+a+2(1+\beta)}.
\end{equation}
We then note that the rescaled function
\begin{equation}\label{resc1}
W_0(X) = w_1(x_1 + y_1 x, y_1 y)
\end{equation}
solves in $\B_{\frac{1}{2}}((0,1))$ a uniformly elliptic PDE with Lipschitz principal part,  bounded drift and scalar term bounded by $ ||f||_{L^{\infty}}\  y_1^2$. From the classical theory we thus have that the  following H\"older estimate holds:
\begin{equation}\label{W1}
|\nabla_x W_0 (X) - \nabla_x W_0((0,1)| \leq C  \left[ \left(\int_{\B_{\frac{1}{2}}((0,1))} W_0^2  dX\right)^{\frac{1}{2}} + ||f||_{L^{\infty}} y_1^2\right] |X-(0,1)|^{\beta}.
\end{equation}
Keeping in mind that
\[
\nabla_x W_0(X) = y_1 \nabla_x w_1(x_1 + y_1 x, y_1 y) = y_1 \nabla_x V(X),
\]
we obtain from \eqref{W1} 
\begin{align}\label{el1}
&|\nabla_x  V(X_1) - \nabla_x  V(X_2)| = |\nabla_x w_1(X_1) - \nabla_x w_1(X_2)|
\\
& \leq  C \left[  \left( \frac{1}{y_1^{n+1}} \int_{\B_{\frac{y_1}{2}}(X_1)}  w_1^2\right)^{\frac{1}{2}} + ||f||_{L^{\infty}} y_1^2 \right]  \frac{|X_1-X_2|^{\beta}}{y_1^{1+\beta}}
\notag\\
& \leq C \left[ \left( \frac{1}{y_1^{n+a+1}} \int_{\B_{\frac{y_1}{2}}(X_1)}  w_1^2 |y|^a \right)^{\frac{1}{2}}  + ||f||_{L^{\infty}} y_1^2\right]  \frac{|X_1-X_2|^{\beta}}{y_1^{1+\beta}} \leq C|X_1-X_2|^{\beta}.
\notag
\end{align}
Note that in the  second inequality in \eqref{el1} we have used that $y\sim y_1$ in $\B_{\frac{y_1}{2}}(X_1)$. Also, in the last inequality we have used the decay estimate \eqref{dc2}.

Suppose now that (2) occurs. We note that, for $i=1, 2$, the function  $w_i(X)\overset{def}{=}  V(X)  - b(x_i) y^{1-a}$ solves the pde \eqref{gl1}  in $B_{\frac{y_i}{2}}(X_i)$. After rescaling as in \eqref{resc1}, and using \eqref{dc2} (which also holds for $w_2$  with $y_1$ replaced by $y_2$), from the classical gradient estimates we obtain  that the following gradient bound is satisfied
\begin{equation}\label{prot2}
|\nabla_x  V (X_i) | = |\nabla_x w_i(X_i)| \leq C y_i^{\beta}.
\end{equation}
The triangle inequality now gives 
\begin{align}\label{y2}
|y_2| & = |X_2 - x_2| \leq |X_2- X_1| + |X_1- x_1| + |x_1 - x_2|
\\
& = |X_2- X_1| + |y_1| + |x_1 - x_2|  \leq 6 |X_2 - X_1|.
\notag
\end{align}
Using \eqref{prot2} and \eqref{y2} we thus find
\[
|\nabla_x V (X_1) -  \nabla_x  V (X_2) | \leq |\nabla_x  V (X_1) | + |\nabla_x  V(X_2) | \leq C |X_1-X_2|^{\beta},
\]
which shows that  $\nabla_x  V \in C^\beta(\overline{\B_{\frac{1}{2}}^+})$ for all $\beta < \min\{1-a,1\}$.  Moreover, the estimate   in \eqref{even2} is seen to hold as well.  

We now establish \eqref{even4} when $f(x,y) \equiv f(x)$. Given $\alpha < \min\{1+a,1\}$, we let $\beta = \alpha -a$. Then, we have that $\beta < \min\{1-a,1\}$. 
We observe that, since $f$ is independent of $y$, for each point $(x_0, 0) \in B_{\frac{1}{2}} \times \{y=0\}$, we have that for $y>0$  the function $h= y^a V_y  - (1-a) b(x_0)$ solves 
\begin{equation}\label{he0}
\operatorname{div}( y^{-a} A(x) \nabla h) = 0,
\end{equation}
where $b$ is as in Lemma \ref{inter}. Given $X_0= (x_0, y_0) \in \B_1^+$,  since $v= V - y^{1-a} b(x_0)$ solves
\begin{equation}\label{ess0}
\operatorname{div}(y^a A(x) \nabla v) = y^a f,
\end{equation}
from the energy estimate applied to $v$  in $\B_{\frac{y_0}{2}}^{+} (X_0)$ it follows that the following inequality holds
\begin{equation}\label{ess1}
\int_{\B_{\frac{y_0}{2}}(X_0)} h^2 y^{-a} \leq \int_{\B_{\frac{y_0}{2}}(X_0)} |\nabla v|^2  y^{a} \leq \frac{C}{y_0^2}   \int_{ \B_{\frac{3y_0}{4}}(X_0)} \left(v^2 + y_0^4 f^2  \right) y^a.
\end{equation}
Using the decay estimate \eqref{dc1} we thus obtain the following bound from \eqref{ess1}
\begin{equation}\label{ess2}
\int_{\B_{\frac{y_0}{2}}(X_0)} h^2 y^{-a}  \leq  C y_0^{n+1+a+2\beta},
\end{equation}
where $C$ also depends on $||f||_{L^\infty}$. Observing now that $h$ solves \eqref{he0}, by rescaling as in \eqref{resc1} we note that the rescaled function solves a uniformly elliptic PDE in $\B_{\frac{1}{2}}((0,1))$.  We can thus apply the  Moser type subsolution estimate  to the rescaled function and then by scaling back we obtain  the following bound 
\begin{align}\label{mos1}
& |h(X_0)| = |y^a V_y(X_0) - (1-a) b(x_0) | \leq C \bigg( \frac{1}{y_0^{n+1-a}} \int_{\B_{\frac{y_0}{4}}(X_0)} h^2 y^{-a} \bigg)^{\frac{1}{2}} \\
& \leq C y_0^{\beta+a} = Cy_0^{\alpha},\
\notag
\end{align}
where in the last inequality in \eqref{mos1} we have used the estimate \eqref{ess2}, and also the fact that in $\B_{\frac{y_0}{2}}(X_0)$ we have that $y \sim y_0$, whereas in the last equality we have used that $\beta+a=\alpha$.

With the  estimate \eqref{mos1} in hand we now show that $y^a V_y$ is in $C^{0,\alpha}$. Again, let $X_1=(x_1, y_1)$ and $ X_2=(x_2, y_2)$ be two points in $\B_{\frac{1}{2}}^+$.  Without loss of generality we assume that  $y_1 \leq y_2$. There are two cases:
\begin{itemize}
\item[(a)] $ |X_1-X_2| \leq  \frac{y_1}{4}$;
\item[(b)] $|X_1 - X_2| \geq \frac{y_1}{4}$.
\end{itemize}
If (a) occurs, then $X_2 \in \B_{\frac{y_1}{4}}(X_1)$ and the function $h_1= y^a V_y - (1-a) b(x_1)$ solves  an equation of the type   \eqref{he0} in  $ \B_{\frac{y_1}{2} }(X_1)$. Again by rescaling as in \eqref{resc1} we note that the rescaled function satisfies a uniformly elliptic PDE with Lipschitz coefficients in $\B_{\frac{1}{2}} ((0,1))$. Arguing as in \eqref{W1}-\eqref{el1} we thus obtain 
\begin{align}\label{hol1}
&|y^a V_y (X_1) - y^a V_y (X_2)| = | h_1(X_1) - h_1(X_2) | 
\\
& \leq \frac{C}{y_1^{\alpha } } \bigg( \frac{1}{y_1^{n+1-a}} \int_{\B_{\frac{y_1}{2}} (X_1)} h^2 y^{-a} \bigg)^{\frac{1}{2}} |X_1-X_2|^{\alpha} \leq C |X_1-X_2|^{\alpha}. 
\notag
\end{align}
We note  that  in the first inequality in \eqref{hol1} we have used that $y \sim y_1$ in $\B_{\frac{y_1}{2}}(X_1)$. Moreover, in the  last inequality in \eqref{hol1} we have used the decay estimate in \eqref{ess2}  with $y_0$ replaced by $y_1$ and also the fact  that $\alpha=\beta+a$.

Suppose now (b) occurs.  In this case we first observe that the estimate   \eqref{mos1} holds when  $X_0$ is replaced by either $X_1$ or $X_2$. More precisely,  we have the following inequality 
\begin{equation}\label{mos2}
|y^a V_y(X_i) - (1-a) b(x_i)|  \leq C y_i^{\alpha},\ \ \ \  \ i=1,2.
\end{equation}
Moreover, from \eqref{y2} we also have that  $|y_2| \leq 6 |X_1 -X_2|$. Consequently, we  obtain 
\begin{align}\label{ev5}
& |y^a V_y (X_1) - y^a V_y (X_2)| \leq |y^a V_y(X_1) - (1-a) b(x_1)|  + |y^a V_y(X_2) - (1-a)b(x_2)|  
\\
&+ (1-a) |b(x_1) - b(x_2)| \leq C y_1^{\alpha} + C y_2^{\alpha} + C|x_1 -x_2|^{\alpha} \leq \tilde C|X_1- X_2|^{\alpha}.\notag\end{align}
We mention that in \eqref{ev5} we have used \eqref{mos2} and the $C^{0,\alpha}$ estimate for the function $b$.  The estimate  \eqref{even4} thus follows.

\end{proof}

We also need the following H\"older estimate for odd solutions which follows from \cite[Theorem 1.6, part 1)]{STV2}.

\begin{theorem}\label{odd2}
Let $V$ be an odd solution to \eqref{Laf} in $\B_1^+$, with $f = 0$. Then for any $\alpha < \min\{1-a,1\}$ we have that  $V  \in C^{0,\alpha}(\overline{\B_{\frac{1}{2}}^+})$ and the following estimate holds
\[
||V||_{C^{\alpha}(\overline{\B_{\frac{1}{2}}^{+})}} \leq C ||V||_{L^{2}(\B_1, |y|^a dX)}.
\]
\end{theorem}


\section{$W^{2,2}$ type estimates and H\"older regularity of $\nabla_x U, y^a U_y$.}\label{initial}

As it is by now well-known, see the works \cite{ACS}, \cite{CSS}, \cite{GP}, \cite{GPG},  two crucial ingredients in the study of the thin obstacle problem \eqref{La} are: a) the monotonicity of  Almgren and Weiss type functionals; and b) the subsequent blow-up analysis. Both a) and b) critically rely on a priori  H\"older estimates for $y^a U_y, \nabla_x U$ similar to those for the case $A=\mathbb{I}$. In this section we establish the $W^{2,2}$ and $C^{1, \alpha}$ estimates that will be essential to our study of \eqref{La}. 
The following is our first result. For brevity, we will use the notation $U_i = U_{x_i}$, $i=1,...,n$, to indicate the tangential partial derivatives.

\begin{theorem}\label{w2}
Let $U$ be a solution to \eqref{La}, with $\psi \in C^{2}$. Then,  the following estimate holds
\begin{equation}\label{est1}
\sum_{i=1}^n \int_{\mathbb{B}_{\frac{1}{2}}^+}  |\nabla U_{i}|^2  y^a dX + \int_{\mathbb{B}_{\frac{1}{2}}^+} ((y^a U_y)_y)^2 y^{-a} dX  \leq  C \int_{\mathbb{B}_1^+}  (U^2 + |\nabla U|^2+1 ) y^a dX,
\end{equation}
where $C = C (||\psi||_{C^2}, ||A||_{C^{0,1}},n)>0$.
\end{theorem}

\begin{proof}
To establish \eqref{est1} we first note that \eqref{La} is equivalent to the minimisation problem 
\begin{equation}\label{min}
\underset{V \in K_{\psi, U}}{\min} \int_{\mathbb{B}^+}  \langle A(x) \nabla V, \nabla V\rangle y^a dX,
\end{equation}
where 
\begin{equation}\label{defk}
K_{\psi,U}= \left\{V \in W^{1,2}( \mathbb{B}^+, y^a dX) \mid V(x,0) \geq \psi (x),\ V=U\ \text{on}\ \S^+_1 \right\}.
\end{equation}
By subtracting off the obstacle $\psi$ from the solution, we  observe that \eqref{La} can be reduced to the following non-homogeneous thin obstacle  problem with zero obstacle
\begin{equation}\label{La1}
\begin{cases}
\var(y^a A(x) \nabla U) = y^a f,\ \ \ \ \ \text{in}\ \mathbb B_1^+,
\\
\min\{U(x,0),-\p_y^a
U(x,0)\}=0 & \text{on}\ B_1,
\end{cases}
\end{equation}
where $f\in L^\infty(\B^+_1)$ and is independent of $y$. 
To study \eqref{La1} we now introduce a one-parameter family of functions $\beta_\ve:\R\to (-\infty,0]$, such that $\beta_\ve(s) \equiv 0$ for $s\ge 0$, $\beta_{\ve}^{'} \geq 0$, 
and $\beta_{\ve}(s)= \ve+\frac{s}{\ve}$, for $s \leq -2 \ve^2$.
In a standard way, \eqref{La1} can now be approximated by solutions to the following penalised problems
\begin{equation}\label{pen}
\begin{cases}
\var(y^a A(x) \nabla U^\ve) = y^a f^{\ve},\ \ \ \ \ \text{in}\ \mathbb B_1^+,
\\
U^\ve= U\ \ \ \ \ \ \ \ \ \ \ \ \ \ \ \text{on}\ \S_{1}^+,
\\
\p_y^a U^{\ve}= \beta_{\ve}(U^\ve),
\end{cases}
\end{equation}
where $f^{\ve}$ is a smooth mollification of $f$ (see for instance  \cite[Chap. 9]{PSU} for the case $a=0$, or also \cite[Sec. 2]{BDGP1}). Using \eqref{cothin} we see that the weak form of \eqref{pen} translates into the equation 
\begin{equation}\label{wkdef1}
\int_{\B_1^+} \langle A(x) \nabla U^{\ve}, \nabla \zeta\rangle y^a dX  + \int_{\B^+_1} \zeta f^\ve y^a = - \int_{B_1} \beta_{\ve}(U^{\ve}) \zeta dx,
\end{equation}
which is requested to hold for any test function $\zeta \in W^{1,2} (\B_1^+, y^a dX) \overset{def} =\{v \in  L^{2}(\B_1^+ y^a dX) \mid \nabla v \in  L^{2}(\B_1^+ y^a dX)\}$  such that $\zeta \equiv 0$ on $\S_1^+$. 
We also note that for the penalised problems \eqref{pen} it follows by a standard  difference quotient type argument that  $y^a |\nabla U^\ve_{x_k}|^2 \in L^{2}_{loc}(\mathbb{B}_1^+)$ for $k=1,2,...,n$. Henceforth, we let $U^{\ve}_k= U^{\ve}_{x_k}$. If we fix $k \in \{1,..., n\}$ and use $\zeta = \eta_{k}= \eta_{x_k}$ as a test function in \eqref{wkdef1}, we obtain
\begin{equation}\label{wkdef0}
\int_{\B_1^+} \left< A(x) \nabla U^{\ve}, \nabla \eta_k \right> y^a dX  + \int_{\B^+_1} \eta_k f^\ve y^a = - \int_{B_1} \beta_{\ve}(U^{\ve}) \eta_k dx.
\end{equation}
If we integrate by parts with respect to $x_k$ in the integrals  $\int_{\B_1^+} \langle A(x) \nabla U^{\ve}, \nabla \eta_k \rangle y^a dX$ and $\int_{B_1} \beta_{\ve}(U^{\ve}) \eta_k dx$ in \eqref{wkdef0}, we find
\begin{equation}\label{ab}
\int_{\B_1^+} \langle A \nabla U^{\ve}_k, \nabla \eta\rangle y^a + \int_{\B_1^+} (b_{ij})_k U^{\ve}_i\eta_j y^a  - \int f^{\ve} \eta_k y^a= - \int_{B_1}\beta_{\ve}^{'}(U^\ve) U^{\ve}_k \eta.
\end{equation}
If we choose 
\begin{equation}\label{eta}
\eta= U^{\ve}_k \tau^2,
\end{equation} 
where $\tau\in C^\infty_0(\B^+_1 \cup B_1)$, keeping in mind that $\beta_{\ve}^{'} \geq 0$ we see that the term in the right-hand side of \eqref{ab} is non-positive. Consequently, using the uniform ellipticity and the  bounds on the derivatives of $A$, Young's inequality, and  by summing over $k=1,...,n$, we obtain the following estimate 
\begin{align}\label{e2}
\sum_{k=1}^n \int_{\B_1^+} |\nabla U^{\ve}_k|^2 \tau^2 y^a \leq  C \int_{\B_1^+}  (|\nabla U^{\ve}|^2 + (f^{\ve})^2 )( \tau^2 + |\nabla \tau|^2) y^a.
\end{align}
It is worth noting here that, although the derivative $\eta_k$ of the function in \eqref{eta} is not a legitimate test function, nevertheless the argument leading to the estimate \eqref{e2} can be justified by first taking incremental quotients of the type
\[
\eta_{k, h}=\frac{\eta(X+he_k)- \eta(X)}{h},
\]
and then letting $h \to 0$. Finally, using the equation \eqref{pen} satisfied by $U^\ve$ we have
\begin{equation}\label{esty}
(y^aU^\ve_y)_y^2 \leq 2 y^{2a} \sum_{i=1}^n (U^\ve_{x_i x_i})^2 \leq 2 y^{2a} \sum_{k=1}^n |\nabla_x U^\ve_k|^2.
\end{equation}
Combining the estimates \eqref{e2} and \eqref{esty} we obtain a bound for $\int ((y^a U^{\ve}_y)_y)^2 \tau^2 y^{-a} dX$ which is uniform with respect to $\ve$. Finally, letting $\ve \to 0$ in such bound and in \eqref{e2}, we obtain the desired estimate \eqref{est1}.

\end{proof}

\begin{remark}
We note that the estimate \eqref{est1} can be localised. Also, taking $U^{\ve} \tau^2$ as a test function in the weak formulation of \eqref{pen},  and using $s \beta_{\ve}(s) \geq 0$, we find
\begin{equation}\label{ip10}
\int |\nabla U^{\ve}|^2 \tau^2 y^a \leq C \int ( (U^{\ve})^2 + (f^{\ve})^2) (\tau^2 + |\nabla \tau|^2)y^a.
\end{equation}
Passing to the limit as $\ve \to 0$ we  deduce that one can get rid of the term involving $\int |\nabla U|^2 y^a dX$ from the right-hand side of the inequality in \eqref{est1}.  The estimate \eqref{est1} also implies that $\p_y^a U_y$ exists as a $L^{2}$ function on the thin set $\{y=0\}$, and moreover we have a.e. on $\{y=0\}$,
\begin{equation}\label{compl}
U\p_y^a U_y =0.
\end{equation}
\end{remark}

Our next result is Theorem \ref{kd} below that provides a quantitative gradient estimate for  solutions to the homogeneous thin obstacle problem which plays a critical role in the proof of the subsequent Theorem \ref{holdery}, see \eqref{an1} below. Theorem \ref{holdery}, in turn, plays a key role in the proof of Theorem \ref{alp}.
We begin with the following version of the Poincar\'e inequality.

\begin{lemma}\label{plem}
Let $v \in W^{1,2} (\B_\rho^{+} \setminus \B_{\rho/2}^+, y^a dX)$. Assume that for some $\gamma>0$ one has 
\[
\mathcal{H}^n\left(\{x \in B_\rho \setminus B_{\rho/2} \mid v(x, 0)=0\}\right) \geq \gamma \rho^n.
\]
Then there exists $C = C(n,a,\gamma)>0$ such that
\[
\int_{\B_{\rho}^+  \setminus \B_{\rho/2}^+ } v^2 y^a \leq C \rho^2 \int_{\B_{\rho}^+ \setminus \B_{\rho/2}^+} |\nabla v|^2 y^a.
\]
\end{lemma}

\begin{proof}
By rescaling, it suffices to assume $\rho =1$. We argue by contradiction and assume that the conclusion of the lemma does not hold. Then there exists a sequence $\{v_k\} \in W^{1,2} (\B_1^{+} \setminus \B_{\frac{1}{2}}^+, y^a dX)$ such that $\mathcal{H}^n\left(\{x \in B_1 \setminus B_{\frac{1}{2}} \mid v_k(x, 0)=0\}\right) \geq \gamma$, $\int_{\B_1^+ \setminus \B_{\frac{1}{2}}^+}   v_k^2 y^a=1$ and $\int_{\B_1^+ \setminus \B_{\frac{1}{2}}^+}   |\nabla v_k|^2 y^a \to 0$. Since $w(X) = y^2$ is an $A_2$-weight, using the extension and compactness theorems in \cite{Ch} it follows that, up to a subsequence, we have $v_k \to v_0$ in $L^{2}(\B_1^+ \setminus \B_{\frac{1}{2}}^+, y^a dX)$, with $\int_{\B_1^+ \setminus \B_{\frac{1}{2}}^+}  v_0^2 y^a =1$. Since $\int_{\B_1^+ \setminus \B_{\frac{1}{2}}^+}   |\nabla v_k|^2 y^a \to 0$, we must have
$v_0 \equiv c$ with $c \neq 0$. By the compactness of the trace operator established in \cite{Ne} it follows that, possibly up to a further subsequence, we have 
\begin{equation*}\label{con0}
\int_{B_1 \setminus B_{\frac{1}{2}}} |v_k - v_0|^2 \to 0.
\end{equation*}
However, since
\[
\int_{B_1 \setminus B_{\frac{1}{2}}} |v_k - v_0|^2 \geq c^2 \mathcal{H}^n\left(\{v_k(x, 0) =0\}\right) \geq \gamma c^2,
\]
this leads to a contradiction, thus establishing the lemma.

\end{proof}

\begin{theorem}\label{kd}
Assume that $a \geq 0$ and let $V$ be the solution to the Signorini problem \eqref{La} in $\B_R^+$ with $\psi \equiv 0$ and $A=\mathbb{I}$. Then there exists $\alpha>0$ such that the following estimate holds for any $0< \rho < R$ 
\begin{equation}\label{kd1}
\int_{\mathbb{B}_{\rho} ^+}   ( y^a V_y - \langle y^a V_y\rangle_{\rho} )^2  y^{-a} \leq  C  \bigg(\frac{\rho}{R}\bigg)^{n+1-a + 2\alpha} \int_{\mathbb{B}_R ^+}   |\nabla V|^2 y^a.
\end{equation}
\end{theorem}

\begin{proof}
The following is a delicate adaption of an argument of Kinderlehrer in \cite{Ki}, where a similar estimate is proven for the  case when $a=0$. By considering $V_R(X)= V(RX)$, we may assume that $R=1$. Furthermore, with this assumption in place, it suffices to establish  \eqref{kd1} for $\rho \leq \frac{1}{8}$. In fact, since we always have
\[
\int_{\mathbb{B}_{\rho} ^+}   ( y^a V_y - \langle y^a V_y\rangle_{\rho} )^2  y^{-a} \leq 2 \int_{\B_{\rho}^+}  |\nabla V|^2 y^a,
\]
if $\rho > \frac{1}{8}$ the estimate \eqref{kd1} is valid in a trivial fashion. 
We first claim that $V$ satisfies the following estimates for any $\rho  < \frac{1}{2}$ 
\begin{equation}\label{hf}
\int_{\mathbb{B}_{\rho}^+} \frac{|\nabla V_{x_i}|^2}{|X|^{n-1 +a} } y^a \leq \frac{C}{\rho^{n+1+a}} \int_{\mathbb{B}_{2\rho}^+  \setminus  \mathbb{B}_{\rho}^+}  (V_{x_i})^2 y^a,\ \ \ \ \ i=1, ... , n,
\end{equation}
and
\begin{equation}\label{hff}
\int_{\mathbb{B}_{\rho}^+}  \frac{|\nabla (y^a V_y)|^2}{|X|^{n-1 -a} } y^{-a}\leq \frac{C}{\rho^{n+1-a}} \int_{\mathbb{B}_{2\rho}^+  \setminus  \mathbb{B}_{\rho}^+}  (V_y)^2 y^a.
\end{equation}
We start with proving \eqref{hf}. As before, we approximate $V$ with the  solutions $V^\ve$ to the following penalised problems with Neumann  condition
\begin{equation}\label{pen1}
\begin{cases}
\var(y^a  \nabla V^\ve) = 0\ \ \ \ \ \text{in}\ \mathbb B_1^+,
\\
\p_y^a V^{\ve}= \beta_{\ve}(V^\ve)\ \ \ \ \ \ \ \text{on}\ B_1,
\end{cases}
\end{equation}
whose weak formulation is 
\begin{equation}\label{wkdef2}
\int_{\B_1^+} \left< \nabla V^{\ve}, \nabla \zeta\right> y^a dX   = - \int_{B_1} \beta_{\ve}(V^{\ve}) \zeta dx,
\end{equation}
for every $\zeta \in W^{1,2} (\B_1^+, y^a dX)$  such that $\zeta \equiv 0$ on $\S_1^+$.
Let $G= \frac{1}{|X|^{n-1+a}}$ and for $0<c<\rho$ consider the truncated functions $G_c = \min \left\{G, \frac{1}{c^{n-1+a}}\right\}$ (at the end we will let $c \to 0$). We notice for subsequent purposes that $\nabla G_c \equiv 0$ in $\overline{\mathbb B}_c$, and that 
\begin{equation}\label{gc}
\var(y^a \nabla G_c) =0\ \  \text{in}\ \B_{1}^+ \setminus \B_{c}^+,\ \ \ \ \ \ \ \p_y^a G_c=0\ \ \text{in}\ (\B_{1} \setminus \B_{c})\cap\{y=0\}.
\end{equation} 
Given $i \in \{1, ..., n\}$, we choose as test function  $\zeta = \eta_{x_i}$ in \eqref{wkdef2}, with
\[
\eta= (V^{\ve})_{x_i} G_c\ \tau^2,
\] 
where $\tau\in C^\infty_0(\B^+_1 \cup B_1)$ is such that $\tau \equiv 1$ in $\mathbb{B}^+_{\rho}$, and $\tau \equiv 0$ outside $\mathbb{B}^+_{2\rho}$.
As in \eqref{wkdef1}-\eqref{ab} above, after substituting such a test function in the weak formulation \eqref{wkdef2}, we integrate by parts with respect to $x_i$ obtaining
\begin{align}\label{t13}
&\int_{\B^+_1} \left(|\nabla( V^{\ve})_{x_i}|^2 G_c  \tau^2  +2 \tau \langle \nabla \tau, \nabla  (V^{\ve})_{x_i}\rangle  (V^{\ve})_{x_i}  G_c  +  V^{\ve}_{x_i} \langle\nabla V^{\ve}_{x_i}, \nabla G_c\rangle \tau^2 \right)  y^a
\\
&= -\int_{B_1} \beta_{\ve}'( V^{\ve}) (V^{\ve})_{x_i}^2 G_c \tau^2   \leq 0.
\notag
\end{align}
Writing the third integral in the left-hand side in \eqref{t13} as $\frac{1}{2} \int_{\B^+_1}  \langle \nabla (V^{\ve}_{x_i})^2 , \nabla G_c\rangle \tau^2 y^a$,
and integrating by parts on the set $\B_{1}^+ \setminus \B_{c}^+$ (where as we have noted the integrand is supported), we obtain 
\begin{align}\label{t14}
& \frac{1}{2} \int_{\B^+_1} \langle \nabla (V^{\ve}_{x_i})^2 , \nabla G_c\rangle\tau^2 y^a
\\
& =  \frac{n-1+a}{2c^{n+a} }\int_{\S_{c}^+}  (V^{\ve}_{x_i})^{2} y^a   - \int_{\B_{1}^+ \setminus \B_{c}^+}   (V^{\ve}_{x_i})^2 \langle \nabla G_c, \nabla \tau\rangle \tau y^a 
\notag
\\
& \geq  - \int_{\B_{1}^+ \setminus \B_{c}^+}   (V^{\ve}_{x_i})^2 |\nabla G_c| |\nabla \tau| \tau y^a.\notag
\end{align}
 Note that in \eqref{t14} we have used both equations in \eqref{gc}. Using  \eqref{t14}  in \eqref{t13},  and also using the numerical inequality $2 \alpha \beta \le \frac 12 \alpha^2 + 2 \beta^2$ to estimate
\begin{align*}
& 2 \int_{\B_{1}^+}   \tau \langle\nabla \tau, \nabla  (V^{\ve})_{x_i}\rangle  (V^{\ve})_{x_i}  G_c\ y^a \le \frac 12 \int_{\B_{1}^+}  |\nabla  (V^{\ve})_{x_i}|^2 G_c \tau^2 y^a
+ 2 \int_{\B_{1}^+} |(V^\ve)_{x_i}|^2 |\nabla \tau|^2 G_c\ y^a, 
\end{align*}
we  finally obtain from \eqref{t13} that the following inequality holds, 
\begin{align}\label{t16}
&\int_{\B_{1}^+}  |\nabla( V^{\ve})_{x_i}|^2 G_c  \tau^2  y^a \\
& \leq  C \int_{\B_{1}^+}   \left( (V^{\ve}_{x_i})^2 |\nabla G_c| |\nabla \tau| \tau +  (V^{\ve}_{x_i})^2 |\nabla \tau|^2 G_c  \right) y^a,
\notag
\end{align}
for some universal $C>0$.
Using now that $|\nabla \tau| \leq \frac{C}{|X|}$, and also that $\nabla \tau$   is supported in $\mathbb{B}_{2 \rho} \setminus \mathbb{B}_{\rho}$, by first  letting $\ve \to 0$ and then $c \to 0$, we conclude from \eqref{t16} that \eqref{hf} is valid.  

We next prove \eqref{hff}.  For that, we crucially use  that $ w^\ve=y^a (V^{\ve})_y$ solves the following problem for the conjugate equation with Dirichlet condition 
\begin{equation}\label{cj}
\begin{cases}
\operatorname{div}(y^{-a} \nabla w^\ve)=0,\ \ \ \ \ \text{in}\ \mathbb B_1^+,
\\
w^\ve(\cdot, 0)= \beta_{\ve}(V^{\ve}), \ \ \ \ \  \text{on}\ B_1.
\end{cases}
\end{equation}
In this respect we observe that, arguing similarly to the proof of Theorem \ref{w2}, one can show that $w^\ve \in W^{1,2}(\B^+_1,y^{-a}dX)$. Once this is done, a computation shows that $w^\ve$ satisfies \eqref{cj}.
Let now $\tilde G= \frac{1}{|X|^{n-1-a}}$, and for $c>0$ also consider $\tilde G_c= \text{min}\ \left( \tilde G, \frac{1}{c^{n-1-a}}\right)$  (as before, we will eventually let $c\to 0$). Using the equation in \eqref{cj}, we now observe that for any $\delta >0$ the following holds
\begin{equation}\label{conju}
\int_{\B^+_1\cap\{y>\delta\}} \operatorname{div}(y^{-a} \nabla w^\ve) \eta=0,
\end{equation}
where  $\eta= w^\ve \tilde G_c \tau^2$ and, as in the proof of \eqref{hf}, $\tau\in C^\infty_0(\B^+_1 \cup B_1)$ is such that $\tau \equiv 1$ in $\mathbb{B}^+_{\rho}$, and $\tau \equiv 0$ outside $\mathbb{B}^+_{2\rho}$.
Integrating by parts in \eqref{conju} we obtain
\begin{align}\label{t17}
& \int_{\B^+_1\cap\{y>\delta\}}  \left( | \nabla w^\ve|^2 \tilde G_c \tau^2  + 2 w^\ve\langle\nabla w^\ve, \nabla \tau\rangle \tau \tilde G_c +  w^\ve \langle\nabla w^\ve, \nabla \tilde G_c\rangle \tau^2 \right) y^{-a}\\
& = - \int_{\B^+_1\cap\{y=\delta\}}  w^\ve_y  w^\ve \tilde G_c \tau^2 y^{-a}
 = \int_{\B^+_1\cap\{y=\delta\}} \Delta_x V^{\ve} w^\ve \tilde G_c \tau^2.
\notag
\end{align}
Note that in the last equality in \eqref{t17} we have used the definition  $w^\ve=y^a (V^{\ve})_y$ and the equation \eqref{pen1} satisfied by $V^{\ve}$. By the continuity of $\Delta_x V^{\ve}, w^\ve$ and $\tau^2$ up to $\{y=0\}$ and Lebesgue dominated convergence theorem, letting $\delta \to 0$ we  deduce from \eqref{t17} 
\begin{align*}
& \int_{\B^+_1}  \left( | \nabla w^\ve|^2 \tilde G_c \tau^2  + 2 w^\ve\langle\nabla w^\ve, \nabla \tau\rangle \tau \tilde G_c +  w^\ve \langle\nabla w^\ve, \nabla \tilde G_c\rangle \tau^2 \right) y^{-a}\\
& = \int_{B_1} \Delta_x V^{\ve} w^\ve \tilde G_c \tau^2 = \int_{B_1}  \Delta_x V^{\ve} \beta_{\ve}(V^{\ve}) \tilde G_c \tau^2.
\notag
\end{align*}
Integrating by parts in the integral in the right-hand side of the latter equality we obtain
\begin{align}\label{t19}
& \int_{\B^+_1}  \left( | \nabla w^\ve|^2 \tilde G_c \tau^2  + 2 w^\ve\langle\nabla w^\ve, \nabla \tau\rangle \tau \tilde G_c +  w^\ve \langle\nabla w^\ve, \nabla \tilde G_c\rangle \tau^2 \right) y^{-a}\\
&= -\int_{B_1} |\nabla_x V^{\ve}|^2 \beta'_{\ve}(V^\ve)  \tau^2  - \int_{B_1} \langle\nabla_x V^{\ve}, \nabla_x ( \tilde G_c \tau^2)\rangle  \p_y^a(V^{\ve})
\notag\\
& \le - \int_{B_1} \langle\nabla_x V^{\ve}, \nabla_x ( \tilde G_c \tau^2)\rangle  \p_y^a(V^{\ve}),
\notag
\end{align}
where in the last inequality we have used $\beta_{\ve}^{'} \geq 0$. By the compactness of the trace operator in \cite{Ne} and the uniform $W^{2,2}$ estimates for $V^{\ve}$ we infer that, possibly passing to a subsequence, $\{\nabla_x V^{\ve}\}, \{\p_y^a V^{\ve}\}$ converge strongly in $L^{2}(B_1,dx) $ to $\nabla_x V, \p_y^a V$ in $B_1$. Therefore, passing to the limit $\ve \to 0$ and using the Signorini condition $V  \p_y^a V\equiv 0$ in $B_1$, which in view of the results in \cite{CSS} implies $\p_y^a V \ \nabla_x V  \equiv 0$ in $B_1$, we conclude that the right-hand side in \eqref{t19} goes to $0$ in the limit as $\ve \to 0$, concluding   
\begin{align}\label{t20}
& \int \left( | \nabla w|^2 \tilde G_c \tau^2  + 2 w\langle\nabla w, \nabla \tau\rangle \tau \tilde G_c +  w \langle\nabla w, \nabla \tilde G_c\rangle \tau^2 \right) y^{-a}\leq 0,
\end{align}
where we have let  $w= y^a V_y$. The third integral in the left-hand side  of \eqref{t20} can be handled similarly to \eqref{t14} using the fact that $\operatorname{div}(y^{-a} \nabla \tilde G_c)= 0$ in $\B_1^+ \setminus \B_c^+$. Arguing as in \eqref{t14}, \eqref{t16} we thus obtain for a universal $C>0$
\begin{align*}
&\int  |\nabla w|^2 \tilde G_c  \tau^2 y^{-a} \leq C \int   \left( w^2 |\nabla \tilde G_c| |\nabla \tau| \tau +  w^2 |\nabla \tau|^2 \tilde G_c\right) y^{-a},
\end{align*}
from which \eqref{hff} follows by letting $c \to 0$. 
We now introduce a notation for the quantities in the left-hand sides of \eqref{hf} and \eqref{hff}, 
\[
I_{i}(\rho)= \int_{\mathbb{B}_{\rho}^+}  \frac{|\nabla V_{x_i}|^2}{|X|^{n-1 +a} } y^{a},\ \ \  i=1, .., n,\ \ \ \ \ \ \ I_{y}(\rho)= \int_{\mathbb{B}_{\rho}^+}  \frac{|\nabla (y^a V_y)|^2}{|X|^{n-1 -a} } y^{-a}.
\]
For later use we observe that there exists a universal constant $C>0$ such that
\begin{equation}\label{sfinalicchio}
I_{y}(\rho) \leq C \sum I_{i}(\rho)
\end{equation} 
For this it suffices to observe that the equation $\var_X(y^a \nabla_X V) = 0$ satisfied by $V$ in $\mathbb B_1^+$ implies
\[
y^{-a} |\nabla (y^a V_y) |^2 \leq C \sum_{i=1}^n y^{a} |\nabla V_{x_i}|^2,
\]
and also that $a\geq 0$ gives $\frac{1}{|X|^{n-1-a} } \leq \frac{1}{|X|^{n-1+a}}$ in $\mathbb B_1^+$.
It is clear that \eqref{sfinalicchio} immediately follows from these observations and the definitions of $I_{i}(\rho)$ and $I_{y}(\rho)$. 

Now, since $\p_y^a V \ \nabla_x V  \equiv 0$ in $B_1$, given any $\rho \in (0,1/4)$ we have
\begin{align*}
& \mathcal{H}^n(B_{2\rho} \setminus B_{\rho}) = \mathcal{H}^n(\{x \in B_{2\rho} \setminus B_{\rho}\mid \p_y^a V(x,0)\ \nabla_x V (x,0) =0\}
\\
& \le \mathcal{H}^n(\{x \in B_{2\rho} \setminus B_{\rho}\mid \nabla_x V (x,0) =0\}) + \mathcal{H}^n(\{x \in B_{2\rho} \setminus B_{\rho}\mid \p^a_y V (x,0) =0\}).
\end{align*}
Therefore, either
\begin{itemize}
\item[(a)] $\mathcal{H}^n(\{x \in B_{2\rho} \setminus B_{\rho}\mid \nabla_x V (x,0) =0\}) \geq \frac{1}{2} \mathcal{H}^n(B_{2\rho} \setminus B_{\rho})$,
\\
must hold, or
\item[(b)] $\mathcal{H}^n(\{x \in B_{2\rho} \setminus B_{\rho}\mid \p_y^a V (x,0) =0\}) \geq \frac{1}{2} \mathcal{H}^n(B_{2\rho} \setminus B_{\rho})$.
\end{itemize}
If (b) occurs then by Lemma \ref{plem}, applied to $y^aV_y$ in $\mathbb{B}_{2\rho}^{+} \setminus \mathbb{B}_{\rho}^+$, we can bound from above the integral in the right-hand side in \eqref{hff} by $C (I_{y} (2 \rho) - I_{y}(\rho))$.
Here, we have used the fact that on the set $\mathbb{B}_{2\rho} \setminus \mathbb{B}_{\rho}$ we have $\frac{1}{|X|^{n-1-a} }\sim \frac{1}{\rho^{n-1-a}}$. On the other hand, if (a) occurs then applying Lemma \ref{plem} to $\nabla_x U$ we obtain that for all $i=1,..., n$ the integral in the right-hand side of  \eqref{hf} can be bounded from above by $C (I_{i}(2 \rho) - I_{i}(\rho))$.
In conclusion, we have shown that for $\rho\in (0,1/4)$ 
either $I_{y}(\rho) \leq C (I_{y}(2 \rho) - I_{y}(\rho))$, or $I_{i}(\rho)\le C (I_{i}(2 \rho) - I_{i}(\rho))$ for $i=1, .., n$. Equivalently, either 
\[
I_{i}(\rho) \leq \frac{C}{C+1} I_{i}(2\rho)\ \text{for}\  i=1, .., n,\ \ \ \text{or}\ \ \ 
I_{y}(\rho) \leq \frac{C}{C+1} I_{y}(2\rho).
\]
Iterating these inequalities on a dyadic sequence of radii $\rho_k= 2^{-k}$ we deduce that with $\gamma= \frac{1}{2} \text{log}_2 (\frac{C}{C+1})$ and for any $\rho \in (0, 1/4)$, either
\begin{equation}\label{pt1}
I_{y}(\rho) \leq C \rho^{\gamma} I_{y}(1/4)
\end{equation}
is true, or 
\begin{equation}\label{pt2}
I_{i}(\rho) \leq C \rho^{\gamma} I_{i}(1/4), \ \text{for}\  i=1,..., n.
\end{equation} 
Suppose that \eqref{pt2} hold. In such case we obtain from \eqref{sfinalicchio}
\begin{equation}\label{sfinal}
I_{y}(\rho) \leq C \sum I_{i}(\rho) \leq C \rho^{\gamma}  \sum I_{i} (1/4) \leq C\rho^{\gamma} \int_{\mathbb{B}_1^+}  |\nabla V|^2 y^a,
\end{equation}
where in the last inequality we have used the energy estimate in \eqref{hf} with the choice $\rho=1/4$. If instead \eqref{pt1} holds,  then we have
\begin{equation*}
I_{y}(\rho) \leq C \rho^{\gamma} I_y(1/4) \leq C \rho^{\gamma} \int_{\mathbb{B}_1^+}  |\nabla V|^2 y^a,
\end{equation*}
where in the second inequality we have applied \eqref{hff}  with $\rho=1/4$. In both cases \eqref{sfinal} holds and, since $n-1-a\ge 1-a\ge 0$, this implies in particular that 
\begin{equation}\label{ssfinal}
\frac{1}{\rho^{n-1 -a}} \int_{\mathbb{B}_{\rho}^+}   |\nabla (y^a V_y)|^2 y^{-a} \leq C \rho^{\gamma}  \int_{\mathbb{B}_1^+}  |\nabla V|^2 y^{a}.
\end{equation}
Combining \eqref{ssfinal} with the weighted Poincar\'e inequality in \cite{FKS} the desired estimate \eqref{kd1} now follows with $\alpha=\frac{\gamma}{2}$ and $R=1$.
 
\end{proof}

\begin{remark}
We stress that by translation the estimate \eqref{kd1} continues to hold for balls centred at any point of the thin set $B_1$. We also note that, although from \cite{CSS} one knows that $y^a V_y \in C^{1-s}$ up to the thin set $B_1$, yet the quantitative estimate \eqref{kd1} does not seem to follow from the results in that paper. 
\end{remark}

We next recall the following real analysis lemma due to Campanato and Morrey, see \cite[Lemma 2.1 on p. 86]{Gia}. It will be needed in the proof of Theorem \ref{holdery}.

\begin{lemma}\label{L:cm}
Let $\vf :[0,\infty)\to [0,\infty)$ be such that $s\le t\Longrightarrow \vf(s) \le \vf(t)$. Suppose that  for every $0<\rho\le R\le R_0$ one has
\begin{equation}\label{ca1}
\vf(\rho) \le A \left[\left(\frac{\rho}{R}\right)^\gamma + \ve\right] \vf(R) + B R^\beta,
\end{equation}
where $A, \alpha, \beta, \ve \ge 0$, with $\beta<\gamma$. There exists $\ve_0 = \ve_0(A,\gamma,\beta)$ such that if $\ve <\ve_0$ one has for $0<\rho\le R\le R_0$
\begin{equation}\label{ca2}
\vf(\rho) \le C \left[\left(\frac{\rho}{R}\right)^\beta \vf(R)+ B \rho^\beta\right],
\end{equation}
where $C =C(A,\gamma,\beta)\ge 0$.
\end{lemma}
The next result asserts the H\"older regularity of $U$ and $y^a U_y$ up to the thin set $\{y=0\}$. 

\begin{theorem}\label{holdery}
Let $U$ be a solution to \eqref{La1} with $a \geq 0$. Then there exists $\beta>0$ such that $U, y^aU_y \in C^{\beta}(\overline{\B_{\frac{1}{2}}^+})$.
\end{theorem}

\begin{proof}
 Without loss of generality we assume again that $A(0,0)= \mathbb{I}$.  The proof is divided into three steps.
 
\medskip

\noindent \emph{Step 1:}  We first show that for any $0<\sigma <1$ there exists $R_{\sigma}>0$ such the following estimate holds for all $0<\rho < R_{\sigma}$ 
 \begin{equation}\label{e13}
\int_{\mathbb{B}_\rho^+} |\nabla U|^2 y^a  \leq C \rho^{n-1- a + 2\sigma} \int_{\mathbb{B}_1^+}\left( |\nabla U|^2 +1\right) y^a,\end{equation} 
where $C= C(n, a, ||f||_{L^{\infty}})>0.$
  Let $0<R<1$ to be fixed sufficiently small subsequently, and denote by  $V$ the minimiser of the energy 
\begin{equation}\label{fr1}
\int_{\mathbb{B}_R^+} |\nabla W|^2 y^a
\end{equation}
over all $W\geq 0$ at $\{y=0\}$ such that $W=U$ on $\S_{R}^+$. From the fact that $V$ minimises \eqref{fr1}, we obtain
\begin{equation}\label{e5}
\int_{\mathbb{B}_R^+} \langle\nabla V, \nabla (V-U) \rangle y^a \leq 0.
\end{equation}
Since $U$ minimises  the energy corresponding to the Euler-Lagrange equation \eqref{La1}, we find
\begin{equation*}
\int_{\mathbb{B}_R^+} (\langle A(x) \nabla U, \nabla (V-U)\rangle  + f (V-U) ) y^a\geq 0.
\end{equation*}
This inequality  can be rewritten as follows 
\begin{align}\label{rew1}
& \int_{\mathbb{B}_R^+} \langle(A(x) - \mathbb{I}) \nabla U, \nabla (V-U)\rangle y^a 
\\
& + \int_{\mathbb{B}_R^+}( \langle \nabla U, \nabla (V-U)\rangle  +f(V-U))y^a \geq 0.
\notag
\end{align}
From \eqref{rew1} we trivially obtain
\begin{align}\label{e4}
& \int_{\mathbb{B}_R^+} \langle(A(x) - \mathbb{I}) \nabla U, \nabla (V-U)\rangle y^a + \int_{\mathbb{B}_R^+} (\langle\nabla V, \nabla (V-U)\rangle +f(V-U)) y^a
\\ 
& \geq \int_{\mathbb{B}_R^+} |\nabla (V-U)|^2 y^a.
\notag
\end{align}
Using \eqref{e5} in \eqref{e4},  and also the fact that the Lipschitz continuity of the matrix $A$ implies
\[
||A- \mathbb{I}||_{L^{\infty}(\mathbb{B}_R^+)}\leq CR,
\]
we find 
\begin{equation}\label{exp}
\int_{\mathbb{B}_R^+} |\nabla (V-U)|^2 y^a \leq CR \int_{\B_R^+} \langle \nabla V, \nabla (V-U)\rangle y^a  +   \int_{\B_R^+}  f (V-U) y^a.
\end{equation}
Using Young's inequality, for every $\delta>0$ the right-hand side in \eqref{exp} can be bounded from the above in the following way 
\begin{align}\label{rew4}
& CR \int_{\B_R^+}  \langle \nabla V, \nabla (V-U)\rangle y^a  +   \int_{\B_R^+}  f (V-U) y^a\\
& \leq  C\delta \int_{\B_R^+}  |\nabla (V-U)|^2 y^a + \frac{C R^2}{\delta} \int_{\B_R^+} |\nabla V|^2 y^a  + \frac{C\delta }{R^2} \int_{\B_R^+} (V-U)^2 y^a  + \frac{CR^2}{\delta}  \int_{\B_R^+}  f^2 y^a   
\notag
\\
& \leq C\delta  \int_{\B_R^+} |\nabla (V-U)|^2 y^a + \frac{C R^2}{\delta} \int_{\B_R^+} |\nabla U|^2 y^a  + \frac{C\delta }{R^2} \int_{\B_R^+} (V-U)^2  y^a + \frac{CR^2}{\delta}  \int_{\B_R^+} f^2  y^a,\notag
\end{align}
where $C>0$ is universal. In the last inequality in \eqref{rew4} we have used $\int_{\B_R^+} |\nabla V|^2 y^a \leq \int_{\B_R^+} |\nabla U|^2 y^a$, which follows from the fact that $V$ minimises the Dirichlet energy in the class of competitors containing $U$. Since $V=U$ on $\{|X|=R\}$, applying to $V-U$ the 
Poincar\'e inequality in \cite{FKS}, we can estimate 
\begin{equation}\label{poto}
\frac{C\delta }{R^2} \int_{\B_R^+} (V-U)^2  y^a \leq C' \delta \int_{\B_R^+} |\nabla (V-U)|^2 y^a,
\end{equation}
where $C'>0$ is another universal constant. 
Using \eqref{poto} in \eqref{rew4}, and finally choosing $\delta$ small enough so that the integral $(C' \delta +C \delta) \int |\nabla (V-U)|^2 y^a$  can be absorbed in the left-hand side of \eqref{exp}, we can finally assert that the following inequality holds for a new $C>0$  
\begin{equation}\label{e7}
\int_{\mathbb{B}_R^+} |\nabla (V-U)|^2 y^a \leq CR^2 \int_{\mathbb{B}_R^+} (|\nabla U|^2  +f^2)y^a.
\end{equation}
Observe now that for any $0<\rho <R$ we have trivially
\begin{align}\label{e8}
&\int_{\mathbb{B}_\rho^+} |\nabla U|^2 y^a \leq 
 C  \int_{\mathbb{B}_\rho^+} ( |\nabla (V-U)|^2   +  |\nabla V|^2) y^a.
\end{align}

It is at this point that we make critical use of the assumption $a \geq 0$ as this limitation is present in \cite[Lemma 3.3]{JP}, which we now use, obtaining
\begin{equation}\label{jp1}
\int_{\mathbb{B}_{\rho}^+} |\nabla V|^2 y^a \leq \left(\frac{\rho}{R}\right)^{n+1-a} \int_{\mathbb{B}_R^+} |\nabla V|^2 y^a \leq  \left(\frac{\rho}{R}\right)^{n+1-a} \int_{\mathbb{B}_R^+} |\nabla U|^2 y^a.
\end{equation}
Inserting \eqref{e7}, \eqref{jp1} in \eqref{e8}, and using the fact that $f \in L^{\infty}$, we find
\begin{align}\label{e101}
&\int_{\mathbb{B}_\rho^+} |\nabla U|^2 y^a \leq C\left(\frac{\rho}{R}\right)^{n+1-a} \int_{\mathbb{B}_R^+} |\nabla U|^2 y^a  + CR^2 \int_{\mathbb{B}_R^+} (|\nabla U|^2 + f^2) y^a
\\
& \leq A\left(\frac{\rho}{R}\right)^{n+1-a} \int_{\mathbb{B}_R^+} |\nabla U|^2 y^a  + AR^2 \int_{\mathbb{B}_R^+} |\nabla U|^2y^a + B R^{n+3+a },
\notag
\end{align}
where $A>0$ is universal and $B >0$ is a universal constant that also depends of the $L^\infty$ norm of $f$. 
Fix now $\sigma \in (0,1)$. Since we can assume without restriction that $R<1$, and since $n+3+a> n-1-a+2\sigma$, it is clear that \eqref{e101} trivially implies the following inequality
\begin{align}\label{e10}
&\int_{\mathbb{B}_\rho^+} |\nabla U|^2 y^a 
 \leq A\left(\frac{\rho}{R}\right)^{n+1-a} \int_{\mathbb{B}_R^+} |\nabla U|^2 y^a  + AR^2 \int_{\mathbb{B}_R^+} |\nabla U|^2y^a + B R^{n-1-a +2 \sigma },
\end{align}

 We now define
\[
\vf(\rho) = \int_{\mathbb{B}_\rho^+} |\nabla U|^2 y^a.
\]
Keeping in mind that $R \leq \delta$, we can express \eqref{e10} in the following form:
\[
\vf(\rho) \le A\left(\left(\frac{\rho}{R}\right)^{\gamma} +\ve\right) \vf(R) + B R^\beta
\]
where $\ve = R^2$, 
\[
\gamma= n+1-a, \quad\ \ \ \text{and}\ \ \ \ \beta= n-1 - a +2 \sigma.
\]
Noting that $0<\beta<\gamma$, by Lemma \ref{L:cm} we infer that, given $\sigma \in (0,1)$, there exists $R_\sigma = R_\sigma(A,n,a)>0$ such that for every $0<\rho\le R\le R_\sigma$ one has
\begin{equation}\label{le1}
\int_{\B_{\rho}^+} |\nabla U|^2 y^a \le C \left[\left(\frac{\rho}{R}\right)^{n-a-1+2\sigma} \int_{\B_{R}^+} |\nabla U|^2 y^a + B \rho^{n-1-a+2\sigma}\right].
\end{equation}
Now by letting $R \to R_\sigma$ we conclude from \eqref{le1} that \eqref{e13} holds. 

\medskip

\noindent \emph{Step 2:} We next prove that there exists $\beta >0$  such that for all $\rho$ small enough one has
\begin{align}\label{an7}
& \int_{\mathbb{B}_\rho^+} ( y^a U_y - \langle y^a U_y\rangle_{\rho} )^2 y^{-a}  \leq C \rho^{n+1-a +2 \beta},
\end{align}
where for a function $f$ we have indicated with $\langle f\rangle_\rho  = \frac{1}{\int_{\mathbb{B}_\rho^+} y^{-a} dX} \int_{\mathbb{B}_\rho^+} f(X) y^{-a} dX$ the integral average of $f$  in $\mathbb{B}_\rho^+$ with respect to the measure $y^{-a} dX$.
To establish \eqref{an7} we apply \eqref{e13} with $\rho = R$ sufficiently small. Again, let $V$ be the minimiser to \eqref{fr1} corresponding to this choice of $R$. We note that in view of the $W^{2,2}$ type estimates in Theorem \ref{w2}, $\p_y^a U$ exists as a $L^2$ function at $y=0$. The triangle inequality now gives for any $0< \rho<R$,
\begin{align*}
& \int_{\mathbb{B}_\rho^+} (y^a U_y - \langle y^a U_y\rangle_{\rho} )^2 y^{-a}
 \leq C \bigg(\int_{\mathbb{B}_\rho^+} (y^a U_y - y^a V_y)^2 y^{-a} +   \int_{\mathbb{B}_\rho^+} (y^a V_y - \langle y^a V_y\rangle_\rho)^2 y^{-a}
\\
& +   \int_{\mathbb{B}_\rho^+} (\langle y^a V_y\rangle_\rho - \langle y^a U_y\rangle_\rho)^2  y^{-a}\bigg) = C \big((I) + (II) + (III)\big),
\end{align*}
where we have slightly abused the notation in writing for instance $\int_{\mathbb{B}_\rho^+} \langle y^a V_y\rangle_\rho y^{-a}$, instead of the more rigorous $\int_{\mathbb{B}_\rho^+} \langle (\cdot)^a V_y\rangle_\rho y^{-a}$. We trivially estimate
\[
(I) \le \int_{\mathbb{B}_\rho^+} |\nabla(U-V)|^2 y^a.
\]
Jensen inequality and the fact that $\int_{\mathbb{B}_\rho^+} y^{-a} dX = C(n,a) \rho^{n+1-a}$ give
\[
(III) \le C \int_{\mathbb{B}_\rho^+} (U_y - V_y)^2 y^{a} \le C \int_{\mathbb{B}_\rho^+} |\nabla(U-V)|^2 y^a.
\] 
Combining estimates, we find
\begin{align}\label{e14}
&\int_{\mathbb{B}_\rho^+} ( y^a U_y - \langle y^a U_y\rangle_{\rho} )^2 y^{-a}
 \leq C(  \int_{\mathbb{B}_\rho^+} ( y^a V_y - \langle y^a V_y\rangle_\rho )^2 y^{-a}  + \int_{\mathbb{B}_\rho^+} |\nabla U - \nabla V|^2 y^a).
\end{align}
To control the first integral in the right-hand side of \eqref{e14} we invoke Theorem \ref{kd} that gives for some $\alpha>0$
\begin{align}\label{an1}
\int_{\mathbb{B}_\rho^+}  ( y^a V_y - \langle y^a V_y\rangle_\rho )^2 y^{-a} \leq C\left( \frac{\rho}{R} \right)^{n+1 -a + 2\alpha} \int_{\mathbb{B}_R^+}  |\nabla V|^2  y^a.
\end{align}
Substituting \eqref{an1} in \eqref{e14} we obtain 
\begin{align}\label{an2}
& \int_{\mathbb{B}_\rho^+} ( y^a U_y - \langle y^a U_y\rangle_{\rho} )^2 y^{-a} \leq  C\left( \frac{\rho}{R} \right)^{n+1 -a + 2\alpha} \int_{\mathbb{B}_R^+}  |\nabla V|^2  y^a + C \int_{\mathbb{B}_\rho^+} |\nabla U - \nabla V|^2 y^a
\\
&\leq C\left( \frac{\rho}{R} \right)^{n+1 -a + 2\alpha} \int_{\mathbb{B}_R^+}  |\nabla U|^2  y^a
+ C R^2 \int_{\B_R^+}   (|\nabla U|^2 + f^2) y^a,
\notag
\end{align}
where in the second inequality we have used \eqref{e7} and $\int_{\mathbb{B}_R^+}  |\nabla V|^2 y^a \leq \int_{\mathbb{B}_R^+}  |\nabla U|^2y^a$, which follows from the fact that $V$ minimises the Dirichlet energy in the class of competitors containing $U$.
Finally, using \eqref{e13}  in \eqref{an2} (with $\rho$ replaced by $R$)  we deduce  that for any $0< \sigma< 1$ there exists $C_\sigma>0$ (depending also on the $W^{1,2}(\B^+_1,y^a dX)$ norm of $U$ and the $L^\infty(\B^+_1)$ norm of $f$) such that 
\begin{align}\label{an4}
& \int_{\mathbb{B}_\rho^+} ( y^a U_y - \langle y^a U_y\rangle_{\rho} )^2 y^{-a}  \leq  C_{\sigma} \bigg( \left( \frac{\rho}{R} \right)^{n+1 -a + 2\alpha}  R^{n-1-a + 2\sigma}
 + R^{2+ n-1 -a + 2\sigma} \bigg).
\end{align}
At this point, we fix $0< \ve< \alpha^{-1}$. Having done this, we now let $\sigma= 1- \ve \alpha$, so that $0<\sigma<1$, and we finally fix a number $\tau$ such that
\begin{equation}\label{constr}
\frac{n+1-a+1-\sigma}{n+1-a+2\sigma} < \tau < \frac{2\alpha-(1-\sigma)}{2(1-\sigma+\alpha)}.
\end{equation}
It is clear that $0<\tau < 1$. Therefore, if we let $R= \rho^{\tau}$ then $0<\rho <R$ and we obtain from \eqref{an4}  
\begin{align}\label{an5}
& \int_{\mathbb{B}_\rho^+} ( y^a U_y - \langle y^a U_y\rangle_{\rho} )^2 y^{-a}  \leq C ( \rho^{n+1-a + 2\alpha - 2\tau(1 - \sigma +\alpha)}   + \rho^{\tau(n+1-a + 2\sigma)}).
\end{align}
The choice \eqref{constr} allows to conclude that
\[
\rho^{2\alpha - 2\tau(1 - \sigma +\alpha)}\le \rho^{1-\sigma},\ \ \ \ \text{and}\ \ \ \ \rho^{\tau(n+1-a + 2\sigma)} \le \rho^{n+1-a + 1-\sigma}.
\]
If therefore we set $\beta= 1-\sigma$, then we can reformulate \eqref{an5} as follows
\begin{align*}
& \int_{\mathbb{B}_\rho^+} ( y^a U_y - \langle y^a U_y\rangle_{\rho} )^2 y^{-a}  \leq C \rho^{n+1-a +2 \beta},
\end{align*}
which establishes the decay estimate \eqref{an7}.

\medskip

\noindent \emph{Step 3:} Using \eqref{an7} from \emph{Step 2} we finally show that $y^a U_y$ is $C^{\beta}$ H\"older continuous up to the thin set $B_1$. 
We  first note that, by translation, the estimate \eqref{an7} continues to hold for balls centred at any point in $B_1$.
Similarly to the argument in the proof of \eqref{hff}, we now observe that $w=y^aU_y$ solves in $\B_1^+$ the following conjugate equation
\[
\operatorname{div}(y^{-a} A(x) \nabla w)=0.
\]
By applying the Campanato type result in \cite[Theorem A.1]{JP} from \eqref{an7} we infer the existence of $h(x) \in C^{\beta}(B_1)$ such that at every $(x,0) \in B_1$ and $0\le r \le 1/2$ one has  
\begin{equation}\label{en0}
\int_{\B_r^+ ((x, 0))}  (w - h(x))^2 y^{-a} \leq C r^{n+1-a + 2\beta}.
\end{equation}
We now observe that for any point $X_1=(x_1, y_1)\in \B_{1/4}^+$ we have $\B_{y_1/2}(X_1) \subset \B_1^+$. The inclusion $\B_{y_1/2}(X_1) \subset \B_1$ is a trivial consequence of the triangle inequality, whereas the inclusion $\B_{y_1/2}(X_1) \subset \{y>0\}$ follows from the fact that if $X = (x,y)\in \B_{y_1/2}(X_1)$, then we have $|y-y_1|\le \frac{y_1}2$, and therefore in particular $y\ge \frac{y_1}2>0$. Let now  $X_1=(x_1, y_1)$ and $X_2=(x_2, y_2)$ be two arbitrary points in $\B_{1/4}^+$.   Without loss of generality we may assume that  $y_1 \leq y_2$.   There are two cases:
\begin{itemize}
\item[(1)] $ |X_1-X_2| \leq  \frac{y_1}{4}$;
\item[(2)] $|X_1 - X_2| \geq \frac{y_1}{4}$.
\end{itemize}
Suppose (1) occurs. In this case, we make use of the  fact that 
 $\tilde w_1 =w - h(x_1) $ solves in $\B_1^+$ the equation 
\begin{equation}\label{prot}
\var(y^{-a} A(x) \nabla \tilde w_1) =0.
\end{equation}
Since the triangle inequality gives $\B_{\frac{y_1}{2}}(X_1)\subset \B^+_{\frac 32 y_1}(x_1,0)$, and since $\frac 32 y_1\le \frac 38 \le \frac 12$, we can apply \eqref{en0} to infer 
\begin{equation}\label{y1}
\int_{\B_{\frac{y_1}{2}} (X_1)} \tilde w_1^2 y^{-a} \leq y_1^{n+1-a+2\beta}.
\end{equation}
Next, we note that the rescaled function
\begin{equation}\label{resc}
W_1(x, y) = \tilde w_1( x_1 + y_1 x, y_1 y)
\end{equation}
solves in $\B_{\frac{1}{2}} (0,1)$ the differential equation  
\begin{equation}\label{rsr1}
\operatorname{tr}(B(x) \nabla^2_{x} W_1) + \partial_{yy} W_1 - \frac{a}{y} \partial_y W_1=0, 
\end{equation}
where $\nabla^2_{x} W_1$ denotes the Hessian of $W_1$ and the Lipschitz matrix-valued function $B(x)= [b_{ij}(x)]_{i,j=1}^n$ is as in \eqref{type}.
Since $\B_{\frac{1}{2} }(0,1)\subset \{X = (x,y)\in \R^{n+1}_+\mid \frac{1}{2} < y < \frac{3}{2}\}$, it is clear \eqref{rsr1}  
is a uniformly elliptic pde with Lipschitz principal part and bounded drift. From the classical theory we  infer that for $X = (x,y)\in \B_{\frac{1}{4}}((0,1))$ the following H\"older estimate holds for $W_1$  
\[
|W_1( X) - W_1( 0, 1)| \leq C|X - (0,1)|^{\beta} \left(\int_{\B_{\frac{1}{2}}((0,1))}  W_1^2\right)^{\frac{1}{2}}.
\]
We note that $W_1( 0,1) = \tilde w_1(X_1)$, and elementary considerations show that if $X = (x,y)\in \B_{\frac{1}{4}}((0,1))$, then $X_2 = (x_2,y_2) = ( x_1 + y_1 x, y_1 y)\in \B_{\frac{y_1}4}(X_1)$. Rewriting the above inequality for $W_1$ in terms of  $\tilde w_1$ we obtain  
\begin{align*}
&|w(X_1) - w(X_2)| = |\tilde w_1(X_1) - \tilde w_1(X_2) | \leq \left(\frac{C}{y_1^{n+1}} \int_{\B_{\frac{y_1}{2}}(X_1)}   \tilde w_1^2\right)^{\frac{1}{2}}   \frac{|X_1 - X_2|^{\beta}}{y_1^{\beta}}\\
&\leq \left(\frac{C}{y_1^{n+1-a}} \int_{\B_{\frac{y_1}{2}}(X_1)}   \tilde w_1^2\ y^{-a}\right)^{\frac{1}{2}}   \frac{|X_1 - X_2|^{\beta}}{y_1^{\beta}} \leq C |X_1-X_2|^\beta.
\end{align*}
In the second inequality above we have used the fact that in the ball $\B_{\frac{y_1}{2}} (X_1)$ one has $y \sim y_1$, whereas in the last inequality we used the estimate \eqref{y1}. 

If instead (2) occurs, then letting $\tilde w_i= w- h(x_i)$ for $i=1, 2$, again we note that $W_i= \tilde w_i( x_i + y_i x, y_i y)$ solves a uniformly elliptic pde of the type \eqref{rsr1} in $\B_{\frac{1}{2}} ((0,1))$. By the classical elliptic estimates applied to $W_i$, and rewritten in terms of  $\tilde w_i$, we obtain   
\begin{equation}\label{prot1}
|\tilde w_i(X_i)|  \leq \left(\frac{C}{y_1^{n+1-a}} \int_{\B_{\frac{y_i}{2}} (X_i)} \tilde w^2\ y^{-a}\right)^{\frac{1}{2}} \leq C y_i^\beta.
\end{equation}
In \eqref{prot1} we have used the fact that the decay estimate \eqref{y1} also holds for  $\tilde w_2$ (when $y_1$ is  replaced by $y_2$).   Also, by an application of triangle inequality we obtain from (2) 
\begin{equation}\label{krt1}
y_2 = |X_2 - (x_2, 0)| \leq |X_2- X_1| + |X_1- (x_1,0)| + |x_1 - x_2|  \leq 6 |X_2 - X_1|.
\end{equation}
Using \eqref{prot1}, \eqref{krt1}  and the $C^{\beta}$ H\"older continuity of $h$,  we conclude that the following inequality holds, 
\begin{align}
&|w(X_1) - w(X_2) | \leq |w(X_1) - h(x_1) | + |h(x_1) - h(x_2) | + |w(X_2) - h(x_2)|
\\
& \leq  C y_1^{\beta} + C  y_2^{\beta} + C|x_1 - x_2|^{\beta}
 \leq C|X_1 - X_2|^{\beta}.\notag
\end{align}
This shows that $w=y^a U_y \in C^{\beta}(\overline{\B_{\frac{1}{2}}^+})$, which, in particular, implies that $\p_y^a U \in L^{\infty}( B_{\frac{1}{2}})$. The fact  that $U \in C^{\beta'}(\overline{B_{\frac{1}{2}}^+})$ for some $\beta'>0$ now follows by a Moser type iteration argument as in \cite{TX}, see also the proof of Theorem 5.1 in \cite{BG}.

\end{proof}

Before proceeding we introduce the \emph{extended free boundary} of $U$, 
\[
\Gamma^\star(U) = \{(x,0)\in B_1\mid U(x,0) = \p^a_y U(x,0) = 0\},
\]
and note that $\Gamma(U) \subset \Gamma^\star(U)$.
We next show that at every point of $\Gamma^\star(U)$ the solution $U$ separates from the obstacle at a rate $1+\sigma$ for every $\sigma< \frac{1-a}2$. This is accomplished by a compactness argument.

\begin{lemma}\label{cl1}
For every $\ve>0$ there exists $\delta>0$ such that if $U$ solves \eqref{La1} in $\B_1^+$, $0 \in \Gamma^*(U)$, $||U||_{L^{\infty}(\B_1^+)} \leq 1$  and $||A- \mathbb{I}||_{C^{0,1}}$, $||f||_{L^{\infty}}  \leq \delta$,
then one can find $V$ that solves \eqref{La1} with $A=\mathbb{I}$ and $f=0$, with  $ ||V||_{L^{\infty}(\mathbb{B}_{3/4}^+)} \leq 1$, $0 \in \Gamma^*(V)$, and such that
\begin{equation}\label{close}
||U-V||_{L^{\infty}(\mathbb{B}_{\frac{1}{2}}^+)} \leq \ve.
\end{equation}

\end{lemma}

\begin{proof}
We argue by contradiction and assume the existence of  $\ve_0>0$ and of a sequence of triplets $\{U^{k}, A^{k}, f^k\}_{k=1}^{\infty}$ such that for every $k\in \mathbb N$ the function $U^k$ solves  
\begin{equation}\label{La10}
\begin{cases}
\var(y^a A^k(x) \nabla U_k) = y^a f_k,\ \ \ \ \ \text{in}\ \mathbb B_1^+,
\\
\min\{U^k(x,0),-\p_y^a
U^k(x,0)\}=0 & \text{on}\ B_1,
\end{cases}
\end{equation}
with
\begin{align}\label{smll}
& ||U^k||_{L^{\infty}(\B_1^+)} \leq1,\ ||A^k - \mathbb{I}||_{C^{0,1}} \leq \frac{1}{k},\   ||f^k||_{L^{\infty}} \leq \frac{1}{k},
\ 0 \in \Gamma^*(U^k),\notag
\end{align}
and such that 
\begin{equation}\label{far1}
||U^k - V||_{L^{\infty}(\mathbb{B}_{\frac{1}{2}}^+)} > \ve_0
\end{equation}
 for every  $V$ that solves \eqref{La1} with $A= \mathbb{I}$, $f=0$, and such that $||V||_{L^{\infty}(\B_{\frac{3}{4}}^+)} \leq 1$. By the H\"older estimates up to the thin set of $V$ and $\p^a_y V$ in Theorem \ref{holdery} we also  have $0 \in \Gamma^*(V)$. 
From the uniform $W^{2,2}$ type estimates for the sequence $U^{k}$  in $\B_{\frac{3}{4}}^+$, which follow from  \eqref{est1} and \eqref{ip10}, as well as the uniform H\"older estimates for $U^k, \p_y^a U^k$ which are a consequence of  Theorem \ref{holdery}, there exists a subsequence which we continue to denote by $U_k$, such that $U_k \to U_0$ uniformly in $\B_{\frac{3}{4}}^+$, $||U_0||_{L^{\infty}(\B_{\frac{3}{4}}^+)} \leq 1$, where $0 \in \Gamma^*(U_0)$, $U_0$ solves \eqref{La1} with $A=\mathbb{I}$, and $f=0$. This  contradicts \eqref{far1} for large enough $k'$s.  The conclusion of the lemma thus follows.

\end{proof}

From the previous result we obtain the following corollary. Before stating it  we make an observation. Suppose that $V$ that solves \eqref{La1} with $A=\mathbb{I}$ and $f=0$, with  $ ||V||_{L^{\infty}(\mathbb{B}_{3/4}^+)} \leq 1$, $0 \in \Gamma^*(V)$. Since $V(0) = 0$ and $V\ge 0$ on the thin set, we infer that it must be $\nabla_x V(0) = 0$. We can thus apply  
the optimal $\frac{3-a}{2}$ decay estimate in \cite[Theorem 6.7]{CSS} to infer the existence of a universal constant $C>0$ such that for every $0 < r< 1/4$
\begin{equation}\label{t200}
||V||_{L^{\infty}(B_r^+)} \leq Cr^{\frac{3-a}{2}}.
\end{equation}
We will use \eqref{t200} momentarily.

\begin{corollary}\label{cr1}
Suppose that  $U$ solves \eqref{La1} in $\B_1^+$, that $||U||_{L^{\infty}(\B_1^+)} \leq 1$ and $0 \in \Gamma^*(U)$. For every $\sigma <\frac{1-a}{2}$ there exist universal $\delta , \lambda \in (0,1/4)$, depending on $\sigma$, such that  if $||A-\mathbb{I}||_{C^{0,1}}, ||f||_{L^{\infty}} \leq \delta$, then 
\begin{equation}\label{gw1}
||U||_{L^{\infty}(\mathbb{B}_{\lambda}^+)} \leq \lambda^{1+\sigma}.
\end{equation}
\end{corollary}

\begin{proof}
Given $\sigma <\frac{1-a}2$ we choose $\lambda < 1/4$ (depending only on the universal constant $C>0$ in \eqref{t200} and on $\sigma$) such that
$C\lambda^{\frac{3-a}{2}} \leq \frac{\lambda^{1+\sigma}}{2}$.
If we let $\ve= \frac{\lambda^{1+\sigma}}{2}$, then from Lemma \ref{cl1} we infer the existence of $\delta = \delta(\ve) = \delta(C,\sigma)>0$ such that if the hypothesis in the lemma are verified, then in correspondence of such $\delta$ there exists $V$ that solves \eqref{La1} with $A=\mathbb{I}$ and $f=0$, with  $ ||V||_{L^{\infty}(\mathbb{B}_{3/4}^+)} \leq 1$, $0 \in \Gamma^*(V)$,
and such that \eqref{close} be true. Since from what has been observed above for such $V$ the estimate \eqref{t200} is in force, the triangle inequality combined with \eqref{close} and \eqref{t200} gives
\[
||U||_{L^{\infty}(\mathbb{B}_{\lambda}^+)} \le ||U-V||_{L^{\infty}(\mathbb{B}_{\lambda}^+)} + ||V||_{L^{\infty}(\mathbb{B}_{\lambda}^+)} \le \ve + C\lambda^{\frac{3-a}{2}} = \frac{\lambda^{1+\sigma}}{2} + C\lambda^{\frac{3-a}{2}}  \le \la^{1+\sigma},
\]
which provides the desired conclusion \eqref{gw1} for $U$.

\end{proof}

From Corollary \ref{cr1} we  obtain the following  almost optimal decay result at any point of the extended free boundary.

\begin{lemma}\label{decy}
Let $U$ be  a solution to the Signorini problem \eqref{La1} such that $||U||_{L^{\infty}(\B_1^+)} \leq 1$ and assume that $0 \in \Gamma^*(U)$. Given any $0< \sigma <\frac{1-a}{2} $ there exists a constant $C=C(n,a,\sigma)>0$ such that for every $r\in (0,1/4)$ one has
\begin{equation}\label{alopt}
||U||_{L^{\infty}(\mathbb{B}_r^+)} \leq C r^{1+\sigma}.
\end{equation}
\end{lemma}

\begin{proof}
Without loss of generality we may assume that $A(0)= \mathbb{I}$. Given $0< \sigma <\frac{1-a}{2}$, let $\la, \delta$ be the universal constants (depending on $\sigma$) whose existence is claimed in Corollary \ref{cr1}. If we let  
$U_r(X)= U(rX)$, then $U_r$ solves \eqref{La1}  corresponding to $A_r(X)= A(rX)$ and $f_r(X)= r^2 f(rX)$, and moreover $0 \in \Gamma^*(U_r)$. Since by the Lipschitz continuity of the matrix-valued function $X\to A(X)$ we have for each $X = (x,y)\in \B^+_1$: $|A_r(X) - I| = |A(rx) - A(0)| \le L r |x|\le L r$, it is clear that there exists $r_0>0$ (depending on $\delta$ above, and therefore on $\sigma$) such that for $r\in (0,r_0]$ the functions $A_r$ and $f_r$ fulfil the constraint 
\begin{equation}\label{kj}
||A_r- \mathbb{I}||_{C^{0,1}}, ||f_r||_{L^{\infty}}  \leq \delta.
\end{equation}
In view of Corollary \ref{cr1}, applied to $U_{r_0}$, this allows to conclude that for every $\la\in (0,1/4)$
\[
||U_{r_0}||_{L^{\infty}(\mathbb{B}_{\lambda}^+)} \leq \lambda^{1+\sigma}.
\]
By rescaling it is clear that it suffices to prove \eqref{alopt} for $U_{r_0}$. Therefore henceforth, to simplify the notation, we drop the subscript $r_0$ and indicate $U_{r_0}, A_{r_0}, f_{r_0}$ with $U, A, f$. With this being said, we now claim that for every $k\in \mathbb N$  the following estimate holds 
\begin{equation}\label{in}
 ||U||_{L^{\infty}(\mathbb{B}_{\lambda^k}^+)} \leq \lambda^{k(1+\sigma)}.
 \end{equation}
Once \eqref{in} is established, the conclusion of the lemma follows by a standard  real analysis argument observing that for any $r< 1/4$ we can find $k\in \mathbb N\cup\{0\}$, such that $\lambda^{k+1} < r \leq \lambda^k$. This gives
 \[
 ||U|_{L^{\infty}(\B_{r}^+)} \leq ||U||_{L^{\infty}(\B_{\lambda^k}^+)} \leq \lambda^{k(1+\sigma)} = \frac{1}{\lambda^{1+\sigma}} \lambda^{(k+1)(1+\sigma)} \leq  \frac{1}{\lambda^{1+\sigma}} r^{1+\sigma}. \]
To achieve \eqref{alopt} it thus suffices to take $C = \frac{1}{\lambda^{1+\sigma}}$. We are left with proving \eqref{in}. We proceed by induction. First, note that from Corollary \ref{cr1} the estimate holds for $k=1$ (for $k=0$ \eqref{in} follows trivially from the assumption $||U||_{L^{\infty}(\B_1^+)} \leq 1$). Assume now that \eqref{in} hold up to some $k\ge 2$. Letting 
\begin{equation}
\tilde U= \frac{U(\lambda^k X)}{\lambda^{k(1+\sigma)}},
\end{equation}
we note that $\tilde U$ solves  a Signorini problem of the type \eqref{La1} with $\tilde A= A(\lambda^k x)$ and $\tilde f= \lambda^{k(1-\sigma)} f(\lambda^k x)$, and also $0 \in \Gamma^*(\tilde U)$. Since $\sigma, \lambda <1$, 
 we observe that thanks to \eqref{kj} and \eqref{in} (which does hold for such $k$ thanks to the inductive assumption),  the hypothesis of Corollary \ref{cr1} is  satisfied. Consequently, \eqref{gw1} holds for $\tilde U$. After scaling back to $U$, which in turn implies the following
 \begin{equation}
 ||U||_{L^{\infty}(\mathbb{B}_{\lambda^{k+1}}^+)} \leq \lambda^{(k+1)(1+\sigma)}.
 \end{equation}
By induction we conclude that \eqref{in} does hold for all $k$, thus completing the proof of the lemma.

\end{proof}

We finally establish the second main regularity result of this section, the (sub-optimal) a priori $C^{\alpha}$ regularity of $\nabla_x U$ up to the thin set. 

\begin{theorem}\label{alp}
Let $a \in [0,1)$ and $U$ be a solution of \eqref{La} with an obstacle $\psi\in C^{2}$. Then $\nabla_x U\in C^{ \alpha}(\overline{\mathbb{B}_{\frac{1}{2}}^+})$ for some $\alpha>0$. 
\end{theorem}

\begin{proof}
By subtracting the obstacle from $U$ we can assume without loss of generality that $U$ solves \eqref{La1}. Since for such case We note that since the $C^\beta$ continuity of $U, y^a U_y$ up to $\{y=0\}$  has already been  established in Theorem \ref{holdery}, we are only left with proving the H\"older continuity of $\nabla_x  U$. Given $X \in \mathbb{B}_{\frac{1}{2}}^+$, we let $d(X)= d(X, \Gamma^*(U))$.  We note that, if such set is non-empty, since there is no point of $\Gamma^\star(U)$
inside $\mathbb{B}_{d(X)} (X) \cap \{y=0\}$, either $\p_y^a U$ or $U$ must vanish identically in this set. Otherwise, the subsets where $U>0$ and $\p_y^a U<0$ would separate the connected set $\mathbb{B}_{d(X)} (X) \cap \{y=0\}$.
By even or odd reflection across $\{y =0\}$ (depending on whether $\p_y^a U\equiv 0$, or $U\equiv 0$) we infer that $U$ solves in $\mathbb{B}_{d(X)}(X)$ 
\begin{equation}\label{eo}
 \operatorname{div}(|y|^a A(x) \nabla U)= |y|^a f.
 \end{equation}
 Moreover, if we fix $\sigma <\frac{1-a}{2}$ then from Lemma \ref{decy} we have
 \begin{equation}\label{k2}
 ||U||_{L^{\infty}(\mathbb{B}_{d(X)}(X))} \leq C d(X)^{1+\sigma}.
 \end{equation}
Using the scaled version of the gradient estimates  in Theorem \ref{even1} or Theorem \ref{odd1},  depending on whether $U$ has been  reflected in an even or odd way across $\{y=0\}$ in $\mathbb{B}_{d(X)}(X)$,   we deduce  from \eqref{k2}  that the following holds,
 \begin{equation}\label{m100}
 |\nabla_x U( X)| \leq C d(X)^{\sigma}.
 \end{equation}
We now take points $X^1, X^2\in \B^+_1$, set $\delta= |X^1 - X^{2}|$, and let $d_i= d(X^i, \Gamma^*(U))$ for $i=1,2$. Without loss of generality we assume that $d_1 \geq d_2$. There exist two possibilities: (a) $\delta \ge \frac{1}{8} d_1$; or, (b) $\delta < \frac{1}{8} d_1$.
If (a) occurs, it trivially follows from \eqref{m100} that
\[
\begin{aligned}
|\nabla_x U(X^1) - \nabla_x U(X^2) | & \leq |\nabla_x U(X^1)| + |\nabla_x U(X^2)| \\
 &\leq C d_1^{\sigma} + C d_2^{\sigma} \leq C \delta^{\sigma}.
\end{aligned}
\]
If instead (b) occurs, then we have $X^2 \in \mathbb{B}_{\frac{d_1}{8}}(X^1)$. As before, we again note that  either  $U$ or $\p_y^a U$ vanishes  identically in $\mathbb{B}_{d_1}(X) \cap \{y=0\}$.  Therefore after an odd or even reflection of $U$ across $\{y=0\}$ in $\mathbb{B}_{d_1}(X)$ (depending on whether $U$ or $\p_y^a U$ vanishes), we obtain that $U$ solves an equation of the type \eqref{eo}. From the  $C^{\alpha}$ estimate of $\nabla_x U$ in  Theorem \ref{even1} or Theorem \ref{odd1} it follows that for some $0<\alpha<\sigma$, the following holds
\begin{align}\label{ip1}
|\nabla_x U(X^1) - \nabla_x U(X^2) | & \leq \frac{C}{d_1^{1+\alpha}} ( ||U||_{L^{\infty}(\mathbb{B}_{\frac{3d_1}{4}}(X))} + d_1^2 ||f||_{L^{\infty}}) \delta^{\alpha} \leq C \delta^{\alpha}.
\end{align}
Note that in the second inequality in  \eqref{ip1} we have used the decay estimate in \eqref{k2} for $X=X_1$  and  the fact that    $\alpha < \sigma$. In both cases a) or b) we obtain for some $\alpha>0$
\[
|\nabla_x U(X^1) - \nabla_x U(X^2) | \leq C |X^1- X^2|^{\alpha},
\]
thus reaching the sought for conclusion.

\end{proof}


\section{Monotonicity formulas}\label{montfor}

In this section we establish a variant of Almgren's monotonicity which is the crucial tool in  the blowup analysis required to establish  the optimal regularity of solutions. We continue to indicate a generic point in the thick space by $X = (x,y)\in \mathbb{R}^{n+1}$, and we set $r=r(X)=|X|$. For notational convenience we will sometimes denote the operator $\var(|y|^a A(x) \nabla)$ by $L_a$. Throughout this section and in the remainder of the paper we will assume, without restriction, that $A(0) = \mathbb I$. This can always be accomplished by a suitable linear transformation of the coordinates. 
We now state our first lemma which can be verified by a standard computation.

\begin{lemma}\label{Lest}
For $r\neq 0$ one has 
\begin{align}\label{Lest1}
L_a r= \var (|y|^{a} A(x) \nabla r)= \frac{n+a}{r}|y|^{a} +O(|y|^{a}).
\end{align}
In particular, $L_a r\in L^{1}(\mathbb{B}_1)$.
\end{lemma}

In the following we will need the function 
\begin{equation}\label{mu}
\mu(X) = \tilde \mu(X) |y|^{a} \overset{def}{=} \frac{\langle A(x)X,X\rangle}{|X|^{2}} |y|^{a} = \langle A(x) \nabla r,\nabla r\rangle |y|^{a}.
\end{equation}
The properties of the function $\tilde \mu(X) = \langle A(x) \nabla r,\nabla r\rangle$ are summarised in \cite[Lemma 4.2]{GG} and will be used in the sequel without further specific reference. Let $U$ be the solution to the thin obstacle problem \eqref{La1} in $\mathbb{B}^+_1$. After an even reflection in $y$ across $\{y=0\}$ we have that $U$ solves in the distributional sense
\begin{equation}\label{ref}
\begin{cases}
\var (|y|^a A(x) \nabla U) = |y|^a f +  2 \p_y^a U \mathscr H^{n} (\{y=0\})\\
	U\ \p_y^a U \equiv 0.
	\end{cases}
\end{equation} 
For any  $r \in (0,1)$ we now define the height function of $U$ in $\mathbb{S}_{r}$ as
\begin{equation}\label{H}
H(r)=\int_{\mathbb{S}_{r}}U^{2}\mu d\sigma,
\end{equation}
where $\mu$ is as in \eqref{mu}.
We also set
\begin{align*}
B(r)= \int_{\B_r}  U^2 |y|^adX.
\end{align*}
The Dirichlet integral of $U$ in $\mathbb{B}_{r}$ is defined as 
\begin{equation}\label{D}
D(r)=\int_{\mathbb{B}_{r}}\langle A(x) \nabla U, \nabla U\rangle |y|^a dX.
\end{equation}
Finally, we denote by	
\begin{equation}\label{I}
I(r)=\int_{\mathbb{S}_{r}} U \langle A \nabla U, \nu \rangle |y|^a d\sigma
\end{equation}
the total energy of $U$ in $\mathbb{B}_{r}$.
We next recall a well-known trace inequality. For its proof we refer the interested reader to e.g.  \cite[Lemma 14.4]{G}.
\begin{lemma}\label{L:trace}
There exists a universal constant $C = C(n,a, \lambda, \Lambda)>0$, such that for $r>0$ and $U\in W^{1,2}(\B,|y|^a dX)$. Then, one has
\begin{equation}\label{t1}
H(r)  \le C \left[\frac 1r B(r) +  r D(r)\right],
\end{equation}
and
\begin{equation}\label{t2}
\frac 1r B(r) \le C \left[H(r) +  r D(r)\right].
\end{equation}
\end{lemma}
The following lemma concerns the first variation of the height function $H$. 
\begin{lemma}\label{Hrder}
The function $H(r)$ is absolutely continuous and for a.e. $r\in(0,1)$ one has 
\begin{equation}\label{h}
H'(r)=2I(r)+\int_{\S_{r}}U^{2}L_a|X|.
\end{equation}
\end{lemma}

\begin{proof}
We follow a by now standard approximation argument that crucially uses the continuity up to the thin set $\{y=0\}$ of the functions $U, \nabla_x U, y^a U_y$, see Theorems \ref{holdery} and \ref{alp}. By first integrating in the region $\B_r \cap \{|y|>\ve\}$, and then letting $\ve \to 0$, by an application of the divergence theorem using the Signorini condition   $U\ \p_y^a U =0$, we can express the height function as the following solid integral
\begin{equation}\label{h1}
H(r)= \int_{\B_r} \var \left( |y|^a    U^2 A \frac{X}{|X|} \right).
\end{equation}
From \eqref{h1} we  obtain
\begin{equation}\label{h2}
H(r) = \int_{\B_r}  U^2 \var\left( |y|^a A \frac{X}{|X|}\right) + 2 |y|^a U \langle A\nabla U,\frac{X}{|X|}\rangle.
\end{equation}
The desired conclusion \eqref{h} now follows from \eqref{h2} by an application of the coarea formula.

\end{proof}

Using \eqref{ref}  and Theorem \ref{alp} again, it is easy to recognise that $I(r)$ and $D(r)$ are related as follows.

\begin{lemma}\label{Ir}
For every $r\in(0,1)$ we have
\begin{align}\label{Ir1}
I(r)=D(r)+\int_{\mathbb{B}_{r}}Uf |y|^a.
\end{align}
\end{lemma}

Following the analysis of the case $a=0$ in \cite{GG}, in order to control the second integral in the right-hand side of \eqref{h} we now introduce some quantities which play a critical auxiliary role. 

\begin{definition}\label{Gdef}
Let $U$ be a solution of \eqref{La1}. Consider the function $G: (0,1] \to (0,\infty)$ defined for any $r\in (0,1]$ by
\begin{align*}
G(r)=\begin{cases}
\frac{\int_{\mathbb{S}_{r}}U^{2}L_a|X|}{\int_{\mathbb{S}_{r}}U^{2}\mu(X)} \hspace{0.5cm}\ \ \ \ \text{if } H(r)\neq 0,
\\
\frac{n+a}{r} \hspace{0.5cm}\ \ \ \ \ \ \ \ \ \ \ \   \text{if } H(r)=0.
\end{cases}
\end{align*}
\end{definition}

\begin{lemma}\label{Glem}
	There exists a universal constant $\beta \ge 0$ such that for any $r\in (0,1)$: 
	\begin{align*}
		\frac{n+a}{r}-\beta \le G(r)\le \frac{n+a}{r}+\beta.
	\end{align*}
\end{lemma}
\begin{proof}
When $r\in (0,1]$ is such that $H(r)=0$ the desired conclusion follows trivially from the definition of $G(r)$.
Since $\tilde{\mu}(X)=O(1)$, and  also  
\begin{align*}
\frac{L r}{\mu}=\big(\frac{n+a}{r}+O(1)\big),
\end{align*}
we infer that there exists a universal constant $\beta \ge 0$ such that
\[
\frac{n+a}{r}-\beta \le \frac{L_a r}{\mu}\le \frac{n+a}{r}+\beta.
\]
This implies
\begin{align*}
\big(\frac{n+a}{r}-\beta \big)\int_{\mathbb{S}_{r}}U^{2} \mu \le \int_{\mathbb{S}_{r}}U^{2} L_a r \le\big(\frac{n+a}{r}+\beta\big) \int_{\mathbb{S}_{r}}U^{2} \mu,
\end{align*}
which concludes the proof.

\end{proof}

Next, with $U$ being the solution of \eqref{La1}, and $G$ as in Definition \ref{Gdef}, following \cite{GG} we introduce the functions $\psi :(0,1]\to(0,\infty)$ and $\sigma:(0,1]\to(0,\infty)$ respectively defined by the Cauchy problems:
\begin{equation}\label{psi}
		\begin{cases}
		 \frac{d}{dr}\log \psi(r)=\frac{\psi'(r)}{\psi(r)}=G(r) \hspace{0.5cm}\ \ \ \  \text{if } r\in(0,1),
		 \\
		 \psi(1)=1,
		\end{cases}
	\end{equation}
	and
\begin{equation}\label{sigma}
	\begin{cases}
	\frac{\sigma'(r)}{\sigma(r)}-\frac{\psi'(r)}{\psi(r)}+\frac{n-1+a}{r}=0 \hspace{0.5cm}\ \ \  \text{if } r\in(0,1),
	\\
	\sigma(1)=1.
	\end{cases}
	\end{equation}

\begin{lemma}\label{psisiglem}
	There exists a universal constant $\beta \ge 0$ such that if $r\in(0,1)$ one has
	\begin{align*}
		\frac{n+a}{r}-\beta\le \frac{d}{dr} \log (\psi(r))\le \frac{n+a}{r}+\beta,
\end{align*}
and therefore
\begin{align*}
e^{-\beta(1-r)}r^{n+a}\le \psi(r)\le e^{\beta(1-r)}r^{n+a}.
\end{align*}
This implies, in particular, that $\psi(0^{+}) = 0$. For the function $\sigma(r)$ we have $\sigma(r)=\frac{\psi(r)}{r^{n-1+a}}$, and so
	\begin{align*}
		e^{-\beta(1-r)}r\le \sigma(r)\le e^{\beta(1-r)}r
	\end{align*}
	for $0<r<1$. In particular, $\sigma(0^{+})=0$.
\end{lemma}

\begin{proof}
	The first inequality is a consequence of Lemma \ref{Glem}. For the first half of the second inequality, we note that integrating the first one over $(r,1)$, we have
\begin{align*}
\log \psi(1)-\log \psi(r)\le (n+a)\big(\log(1)-\log(r\big)+\beta(1-r) \implies \psi(r)\ge e^ {-\beta(1-r)}r^{n+a}.
\end{align*}
Same steps to obtain the second-half of the second inequality. For the third one, we observe that 
\begin{align*}
\log(\sigma(1))-\log(\sigma(r))=\log(\psi(1))-\log(\psi(r))-(n-1+a)\big(\log(1)-\log(r)\big),
\end{align*}
which implies $\log (\sigma(r))=\log (\psi(r)r^{-(n-1+a)})$ and thus $ \sigma(r)=\psi(r)r^{-(n-1+a)}$.

\end{proof}

\begin{lemma}
	There exists a universal constant $r_{0}$ such that the function $r\mapsto\sigma(r)$ is increasing on $(0,r_{0})$.
\end{lemma}

\begin{proof}
By Lemma \ref{psisiglem}	we know that
\begin{align*}
\frac{\sigma'(r)}{\sigma(r)}=\frac{\psi'(r)}{\psi(r)}-\frac{n-1+a}{r}=G(r)-\frac{n-1+a}{r}\ge \frac{n+a}{r}-\beta - \frac{n-1+a}{r}=\frac{1}{r}-\beta.
\end{align*}
If we take $r_{0}<\beta^{-1}$, we obtain $\frac{\sigma'(r)}{\sigma(r)}\ge 0$.

\end{proof}

We now note that, if we consider the numbers: $\alpha^{-}=\liminf_{r\to 0^{+}}\frac{\sigma(r)}{r}$ and $ \alpha^{+}=\limsup_{r\to 0^{+}}\frac{\sigma(r)}{r}$, then we obviously have 
\begin{align*}
	0<e^{-\beta}\le \alpha^{-}\le \alpha ^{+}\le e^{\beta}.
\end{align*}
The following lemma will be needed in the proof of optimal regularity of solutions to \eqref{La1}.
\begin{lemma}\label{ese}
One has for $r\in(0,1)$
\begin{align*}
\left|\frac{\sigma(r)}{r}-\alpha^{\pm}\right|\le \beta e^{\beta}r.
\end{align*}
In particular, we  have $\alpha^{+}=\alpha^{-}$ and thus, in particular, it exists
\begin{align*}
\alpha\overset{def}{=}\lim_{r\to 0^+}\frac{\sigma(r)}{r}>0.
\end{align*}
\end{lemma}
\begin{proof}
We start with the preliminary observation
\begin{align*}
& \left|\frac{d}{dr}\log \frac{\sigma(r)}{r}\right|=\left|\frac{d}{dr}\big(\log (\sigma(r))-\log r\big)\right|=\left|\frac{\sigma'(r)}{\sigma(r)}-\frac{1}{r}\right|\\
&=\left|\frac{\psi'(r)}{\psi(r)}-\frac{n-1+a}{r}-\frac{1}{r}\right|=\left|\frac{\psi'(r)}{\psi(r)}-\frac{n+a}{r}\right|\le \beta,
\end{align*}
where in the latter inequality we have used Lemma \ref{psisiglem}. If we define $h(r)=\log \frac{\sigma(r)}{r}$, then by Lemma \ref{psisiglem} and the fact $r\in (0,1)$, we have
\begin{align*}
\left|\frac{d}{dr}\frac{\sigma(r)}{r}\right|= |h'(r)|e^{h(r)}=\left|\frac{d}{dr} \log \frac{\sigma(r)}{r}\right|\frac{\sigma(r)}{r}\le \beta \frac{\sigma(r)}{r}\le \beta e^{\beta(1-r)}\le \beta e^{\beta}.
	\end{align*}
If we set $g(r)=\frac{\sigma(r)}{r}$, and fix $0<\varepsilon<r<1$, we have
\begin{align*}
|g(r)-g(\varepsilon)|=|\int_{\varepsilon}^{r}g'(\tau)d\tau|\le \int_{\varepsilon}^{r}\beta e^{\beta} d\tau =\beta e^{\beta}(r-\varepsilon)\le \beta e^{\beta}r,
\end{align*}
which implies $g(\varepsilon)-\beta e^{\beta}r\le g(r)\le g(\varepsilon)+\beta e^{\beta}r$. Taking $\liminf_{\varepsilon\to 0^{+}}$ and $\limsup_{\varepsilon\to 0^{+}}$ in the above inequalities, we obtain
\begin{align*}
\alpha^{\pm}-\beta e^{\beta}r\le g(r)\le \alpha^{\pm}+\beta e^{\beta}r \implies |\frac{\sigma(r)}{r}-\alpha^{\pm}|\le \beta e^{\beta}r, 
\end{align*}
for $r\in (0,1)$. We conclude observing that 
\begin{align*}
0\le \alpha^{+}-\alpha^{-}\le (\alpha^{+}-g(r))-(\alpha^{-}-g(r))\le |\alpha^{+}-g(r)|+|\alpha^{-}-g(r)|\le 2\beta e^{\beta}r\to 0 
\end{align*}
as $r\to 0^{+}$.
	
\end{proof}

In the subsequent steps we will need the following two lemmas.

\begin{lemma}\label{estSr}
Let U be the solution of \eqref{La1} with $U(0)=0$. Then,
\begin{align}\label{estSr1}
\int_{\mathbb{S}_{r}}U^{2}\mu \le C r \int_{\mathbb{B}_{r}} \langle A\nabla U, \nabla U\rangle |y|^{a} + Cr^{n+4+a},
\end{align}
where $C>0$ is a universal constant depending on $||f||_{\infty}$.
\end{lemma}

\begin{proof}
Note that from $f\in L^\infty(\B_1)$ we deduce that $U$ is a supersolution to
\begin{equation}\label{t1}
\var (|y|^a A(x) \nabla  U) \leq C |y|^a.
\end{equation}
Keeping \eqref{mu} in mind, we set 
\[
L(r)= \int_{\S_r}   U  \mu. 
\]
We have  
\[
L'(r) = \frac{n+a}{r} L(r) + O(1) L(r) + \int_{\S_r} U_{\nu} \tilde \mu.
\]
This can be further rewritten as
\begin{align}\label{rew}
& L'(r) = \left(\frac{n+a}{r} + O(1)\right) L(r)  + \int_{\S_r} |y|^a < A\nabla U, \nu> + \int_{\S_r} |y|^a < \nabla U, \tilde \mu \nu - A\nu>
\\
&= \left(\frac{n+a}{r} + O(1)\right) L(r) + \int_{\B_r} \var(|y|^a A\nabla U) + \int_{\S_r}   |y|^a< \nabla U, \tilde \mu \nu - A\nu>
\notag
\\
& \leq \left(\frac{n+a}{r} +O(1)\right) L(r) + Cr^{n+1+a}  + \int_{\S_r}   |y|^a< \nabla U, \tilde \mu \nu - A\nu>,
\notag
\end{align}
where in the last inequality we have used \eqref{t1}.
  Now is easily checked that the vector $k= \tilde \mu \nu- A\nu$ is tangential to the sphere $\S_r$ and thus by applying divergence theorem on the sphere, we deduce from \eqref{rew}  that the following holds,
\begin{align}\label{sup}
&L'(r) \leq \frac{n+a}{r} L(r) + O(1) L(r) + Cr^{n+1+a}  + \int_{\S_r}  U \operatorname{div}_{\S_r} (k) |y|^a   + \int_{\S_r} U <\nabla |y|^a, k>\\
& \leq  \frac{n+a}{r} L(r) + O(1) L(r) + Cr^{n+1+a}. \notag\end{align}
Over here we used the fact that $\var_{\S_r} k = O(1)$ since $A$ is Lipschitz and $A(0)=\mathbb{I}$ and also that
\[
<\nabla y^a, k> = a y^{a-1} (\tilde \mu -1) \frac{y}{r} \leq C y
\]
since $(\tilde \mu - 1) = O(r)$. 
From \eqref{sup} we obtain  that with 
\[
L_0(r) = \frac{L(r)}{r^{n+a}}
\]
we have that for some universal $C>0$
\[
 r \to e^{- Cr} L_0(r) - C r^2
 \]
 is non-increasing from which it  follows that
 \begin{equation}\label{sup1}
 \frac{1}{r^{n+a}} \int_{\S_r}   U  \mu  \leq C U(0) +  Cr^2,\ r \in (0,1).
 \end{equation}
Using the super mean value inequality in  \eqref{sup1}, one can argue as in the proof of \cite[Lemma 2.13]{CSS} to deduce the validity of \eqref{estSr1}.

\end{proof}

\begin{corollary}\label{estSr2}
	Let $U$ be the solution of \eqref{La1} such that $U(0)=0$. Then,
	\begin{align}\label{estSr3}
	   \int_{\mathbb{B}_{r}}U^{2}|y|^{a} \le C r^{2} \int_{\mathbb{B}_{r}} \langle A \nabla U, \nabla U \rangle |y|^{a}+Cr^{n+a+5}.
	\end{align}
\end{corollary}

\begin{proof}
Keeping in mind that $\mu(X)\le \lambda^{-1} |y|^{a}$, integrating \eqref{estSr1} between $(0,r)$ we obtain 
\begin{align*}
\lambda\int_{\mathbb{B}_{r}}U^{2}|y|^{a}\le\int_{0}^{r}\int_{\mathbb{S}_{\rho}}U^{2}\mu d\rho\le C \int_{0}^{r}\rho \int_{\mathbb{B}_{\rho}} \langle A\nabla U, \nabla U\rangle |y|^{a}d\rho + C\int_{0}^{r}\rho^{n+a+4}d\rho.
	\end{align*}
	By  integrating by parts, we then observe that
	\begin{align*}
		\int_{\mathbb{B}_{r}}U^{2}\mu &\le C'r^{2} \int_{\mathbb{B}_{r}}\langle A\nabla U, \nabla U\rangle |y|^{a}-C'\int_{0}^{r}\rho^{2}\int_{\mathbb{S}_{\rho}}\langle A\nabla U, \nabla U\rangle|y|^{a} d\rho+ C''r^{5+a+n}=\\
		&=C'r^{2} \int_{\mathbb{B}_{r}}\langle A\nabla U, \nabla U\rangle |y|^{a}-C'\int_{\mathbb{B}_{r}}|X|^{2}\langle A\nabla U, \nabla U\rangle|y|^{a} d\rho+ C''r^{5+a+n}\le\\
		&\le C'r^{2} \int_{\mathbb{B}_{r}}\langle A\nabla U, \nabla U\rangle |y|^{a}+ C''r^{5+a+n}\le Cr^{2} \int_{\mathbb{B}_{r}}\langle A\nabla U, \nabla U\rangle |y|^{a}+ Cr^{5+a+n}.
	\end{align*}
	The conclusion thus follows.
	
\end{proof}

Given $\delta \in (0,1)$ and a universal constant $r_{0}>0$ (which will also depend on $\delta$), we now introduce the sets:
\begin{align}\label{lambda}
\Lambda_{r_{0}} & =\{r\in(0,r_{0}) \mid H(r)>\psi(r)r^{3+\delta} \},
\\
\Gamma_{r_{0}} & =\{r\in(0,r_{0}) \mid H(r)>e^{-\beta}r^{3+\delta+n+a} \},
\notag
\end{align} 
where $\beta\ge 0$ is the constant in Lemma \ref{psisiglem}. 

\begin{lemma}
One has the inclusion $\Lambda_{r_{0}} \subset \Gamma_{r_{0}}$. In particular, $H(r)\not= 0$ for every $r\in \Lambda_{r_{0}}$.
\end{lemma}

\begin{proof}
	By Lemma \ref{psisiglem}, let $r\in \Lambda_{r_{0}}$, we have
\begin{align*}
H(r)>\psi(r)r^{3+\delta}\ge e^{-\beta(1-r)}r^{3+\delta+n+a}\ge e^{-\beta}r^{3+\delta+n+a},
\end{align*}
which implies $r\in \Gamma_{r_{0}}$. The second part of the statement is an obvious consequence of the first one.
	
\end{proof}

\begin{lemma}\label{estHr}
	Assume that  $U(0)=0$. There exists a universal $r_{0}>0$, depending also on $\delta \in (0,1)$, such that:
	\begin{align}\label{estHr1}
	H(r)\le 2 Cr D(r) \hspace{0.5cm} r\in \Gamma_{r_{0}}
	\end{align}
	where $C$ is the same as in \eqref{estSr1}.
\end{lemma}

\begin{proof}
	By \eqref{estSr1} we get $H(r)\le C r D(r)+ Cr^{4+a+n}$. Then, if $r\in \Gamma_{r_{0}}$, we get:
	\begin{align*}
		r^{4+a+n}=r^{3+\delta+a+n}r^{1-\delta}\le e^{\beta}H(r)r^{1-\delta} \implies H(r)\le C r D(r)+ C'r^{1-\delta}H(r).
	\end{align*}
	By taking $r<r_{0}<\big(\frac{1}{2 C'}\big)^{\frac{1}{1-\delta}}$, we note that  $C'r^{1-\delta}\le \frac{1}{2}$ which in turn implies that
	\begin{align*}
		H(r)\le C r D(r) + \frac{H(r)}{2} \implies \frac{H(r)}{2}\le C r D(r) \implies H(r)\le 2 C r D(r).
	\end{align*}
\end{proof}

\begin{corollary}\label{rn3a}
     Suppose that $U(0)=0$. There exists a universal constant $r_{0}>0$, depending also on $\delta \in (0,1)$, such that:
     \begin{align}\label{rn3a1}
     	r^{n+3+a}\le 2 Ce^{\beta}r^{1-\delta}D(r)  \hspace{ 0.5cm} r\in \Gamma_{r_{0}}.
     \end{align}
\end{corollary}

\begin{proof}
By Lemma \ref{estHr} , we have for $r\in \Gamma_{r_{0}}$ such that:
  \begin{align*}
  	r^{3+n+a}=r^{-\delta}r^{3+n+a+\delta}\le e^{\beta}r^{-\delta}H(r)\le 2C e^{-\beta}r^{1-\delta}D(r).
\end{align*}
\end{proof}

\begin{lemma}\label{ID}
	Let $U(0)=0$. There exists a universal $r_{0}>0$, depending on $\delta \in (0,1)$ and $||f||_{\infty}$, such that if $r\in \Gamma_{r_{0}}$ then:
	\begin{align}\label{ID1}
		I(r)\ge \frac{D(r)}{2}
	\end{align}
\end{lemma}

\begin{proof}
 By \eqref{Ir1} we need to prove
 \begin{align*}
 	\left|\int_{\mathbb{B}_{r}}Uf|y|^{a}\right|\le \frac{D(r)}{2}.
 \end{align*}
 Note that, 
 \begin{align*}
 	D(r)=I(r)-\int_{\mathbb{B}_{r}}Uf |y|^a\le I(r)+|\int_{\mathbb{B}_{r}}Uf |y|^a|\le I(r)+\frac{D(r)}{2}.
 \end{align*}
 By Cauchy-Schwartz, since $f\in L^{\infty}$:
 \begin{align*}
 	|\int_{\mathbb{B}_{r}}Uf|y|^{a}|\le C \int_{\mathbb{B}_{r}}U|y|^{a/2}|y|^{a/2}\le C \big(\int_{\mathbb{B}_{r}}|y|^{a}\big)^{\frac{1}{2}}\big(\int_{\mathbb{B}_{r}}|U|^{2}|y|^{a}\big)\le Cr^{\frac{a}{2}} r^{\frac{n+1}{2}}\big(\int_{\mathbb{B}_{r}}U^{2}|y|^{a}\big)^{\frac{1}{2}}.
 \end{align*}
 Now, by Corollary \ref{estSr2} we have:
 \begin{align*}
 	\int_{\mathbb{B}_{r}}U^{2}|y|^{a}\le C r^{2}\int_{\mathbb{B}_{r}}\langle A \nabla U, \nabla U \rangle |y|^{a}+Cr^{5+a+n}.
 \end{align*}
 Thus,
 \begin{align*}
     |\int_{\mathbb{B}_{r}}Uf|y|^{a}|&\le Cr^{\frac{a}{2}} r^{\frac{n+1}{2}}\big( r^{2}\int_{\mathbb{B}_{r}}\langle A \nabla U, \nabla U \rangle |y|^{a}+r^{5+a+n})^{\frac{1}{2}}\le C r^{\frac{n+1+a}{2}}\big(r[D(r)]^{\frac{1}{2}}+r^{\frac{5+a+n}{2}}\big)=\\
     &=C r^{\frac{n+1+a}{2}}\big(r[D(r)]^{\frac{1}{2}}+r^{\frac{5+a+n}{2}}\big)=C\big(r^{\frac{n+3+a}{2}}[D(r)]^{\frac{1}{2}}+r^{n+a+3}\big).
 \end{align*}
 \begin{align}\label{rn3a2}
 	\implies |\int_{\mathbb{B}_{r}}Uf|y|^{a}|\le C\big(r^{\frac{n+3+a}{2}}[D(r)]^{\frac{1}{2}}+r^{n+a+3}\big).
 \end{align}
 Notice, now, that $\forall c_1, c_2 >0$ and $\forall \varepsilon>0$, we have:
 \begin{align}
 	&\big(\sqrt{\varepsilon c_1}-\sqrt{\frac{c_2}{\varepsilon}}\big)^{2}=\varepsilon c_1+\frac{c_2}{\varepsilon}-2\sqrt{c_1 c_2}\ge 0,\\ & \implies \varepsilon c_1+\frac{c_2}{\varepsilon}\ge 2\sqrt{c_1 c_2} \ge \sqrt{c_1 c_2},\notag \\ & \implies \sqrt{c_1c_2}+c_2\le \varepsilon c_1+(\frac{1}{\varepsilon}+1)c_2.\notag
 \end{align}
This means that $\forall \varepsilon >0$ we have:
 \begin{align*}
 	|\int_{\mathbb{B}_{r}}Uf|y|^{a}|\le C \varepsilon D(r)+C(\frac{1}{\varepsilon}+1)r^{n+a+3}.
 \end{align*}
 By \eqref{rn3a1} we have that if $r\in \Gamma_{r_{0}}$, by letting $\ve=\frac{1}{4C}$:
 \begin{align*}
 	|\int_{\mathbb{B}_{r}}Uf|y|^{a}|&\le C \varepsilon D(r)+C(\frac{1}{\varepsilon}+1)r^{n+a+3}\le C \varepsilon D(r)+ C(\frac{1}{\varepsilon}+1)2C_{1}e^{\beta}r^{1-\delta}D(r)=\\
 	&=\frac{D(r)}{4}+2e^{\beta}C C_{1}(4C+1)r^{1-\delta}D(r).
 \end{align*}
 We now let $r_{0}=\big(\frac{1}{4} \frac{1}{2e^{\beta}C_{1}C(4C+1)}\big)^{\frac{1}{1-\delta}}$. Thus, 
 \begin{align*}
 	 r<r_{0} \implies 2e^{\beta}C_{1}C(4C+1)r^{1-\delta}\le \frac{1}{4}.
 \end{align*}
 Therefore  for every $r\in \Gamma_{r_{0}}$, it follows that
 \begin{align*}
 	\left|\int_{\mathbb{B}_{r}}Uf |y|^{a}\right|\le \frac{D(r)}{2}.
 \end{align*}
\end{proof}

Before proceeding further, we need to compute the derivative of the total energy of $U$ introduced in \eqref{I}. We need the following lemma which can be verified by a  standard  computation keeping in mind the definition of the function $\tilde \mu$ in \eqref{mu}.
\begin{lemma}\label{divZ}
Consider the vector field $Z\overset{def}{=}\frac{A(x)X}{\tilde{\mu}(X)}$. We have 
\begin{align*}
\partial_i Z_j = \delta_{ij} + O(r),\ \ \ \ \ \ \ \ \ \ \ \ \var Z=(n+1)+O(r).
\end{align*}
\end{lemma}

Our next result concerns the first variation of $D(r)$.

\begin{theorem}\label{Dder}
	\begin{align}\label{Dder1}
		D'(r)=2\int_{\mathbb{S}_{r}}\frac{(\langle A(x)\nabla U, \nu\rangle)^{2}}{\tilde \mu } |y|^a+\left(\frac{n-1+a}{r}+O(1)\right)D(r)-\frac{2}{r} \int_{\mathbb{B}_{r}} \langle Z, \nabla U \rangle f |y|^{a}.
	\end{align}
\end{theorem}
\begin{proof}
First, by the coarea formula we see that 
\begin{equation}\label{d1}
D'(r)= \int_{\S_r} \langle A\nabla U, \nabla U\rangle |y|^a.
\end{equation}
Since $<Z, \nu>= r$ on $\S_r$, by an application of the divergence theorem we deduce   
\begin{equation}\label{d2}
D'(r) = \int_{\B_r} \var(|y|^a \langle A\nabla U, \nabla U\rangle Z).
\end{equation}
We now use the following Rellich type  identity
\begin{align}\label{re1}
& \var (|y|^a \langle A\nabla U, \nabla U\rangle Z)= 2 \var (|y|^a \langle Z, \nabla U\rangle A\nabla U)
\\
& = \var(Z) |y|^a \langle A\nabla U, \nabla U\rangle  + |y|^a Z_l \partial_{l} a_{jk} \partial_{j}U \partial_k U
\notag
\\
&+ Z_{n+1} a|y|^{a-2} y \langle A\nabla U, \nabla U\rangle - 2 \langle Z, \nabla U\rangle \var(|y|^a A \nabla U) - 2 \partial_i Z_k a_{ij} \partial_j U \partial_k U.
\notag
\end{align}
Using  the equation \eqref{ref} satisfied by $U$ the identity in  \eqref{re1} and Lemma \ref{divZ}, we obtain from \eqref{d2} that the following holds,
\begin{align}\label{d4}
&D'(r) = 2 \int_{\mathbb{S}_{r}}\frac{(\langle A(x)\nabla U, \nu\rangle)^{2}}{\tilde \mu } |y|^a +\left(\frac{n-1+a}{r}+O(1)\right)D(r)-\frac{2}{r} \int_{\mathbb{B}_{r}} \langle Z, \nabla U \rangle f |y|^{a}\\
& - \frac{2}{r} \int_{\B_r \cap \{y=0\}}\langle Z, \nabla U \rangle \p_y^a U.
\notag
\end{align}
We note that the formal computation leading to  \eqref{d4} can again be justified by a limiting type argument as before using the continuity of $U, \nabla_x U, y^a U_y$ up to $\{y=0\}$ as well as the $W^{2,2}$ type estimates for $U$.  Finally by noting that at $\{y=0\}$,
\begin{equation}\label{d5}
\langle Z, \nabla  U\rangle \p_y^a U= \langle x, \nabla_x U\rangle \p_y^a U \equiv 0,
\end{equation}
thanks to the complementarity condition in  \eqref{ref}, we thus conclude    by using \eqref{d5} in \eqref{d4} that \eqref{Dder1} holds. 

\end{proof}

\begin{theorem}\label{Ider}
	Let $U$ be the solution of \eqref{La1}. Then, for a.e. $r\in (0,1)$ we have
\begin{align}\label{Ider1}
I'(r)&=2\int_{\mathbb{S}_{r}}\frac{(\langle A(x)\nabla U, \nu\rangle)^{2}}{\tilde{\mu}} |y|^a+\left(\frac{n-1+a}{r}+O(1)\right)I(r)+\int_{\mathbb{S}_{r}}Uf |y|^a\\
&-\left(\frac{n-1+a}{r}+O(1)\right)\int_{\mathbb{B}_{r}}Uf |y|^a -\frac{2}{r}\int_{\mathbb{B}_{r}}\langle Z,\nabla U \rangle f |y|^{a}\notag.
\end{align}
\end{theorem}

\begin{proof}
By Lemma \ref{Ir} we have that $I(r)=D(r)+\int_{\mathbb{B}_{r}}Uf |y|^{a}$ and thus by \eqref{Dder1}:
   \begin{align*}
   	I'(r)=D'(r)+\int_{\mathbb{S}_{r}}Uf|y|^a&=2\int_{\mathbb{S}_{r}}\frac{(\langle A(x)\nabla U,\nu \rangle)^{2}}{\tilde{\mu}} |y|^a +(\frac{n-1+a}{r}+O(1))D(r)-\\
   	&-\frac{2}{r}\int_{\mathbb{B}_{r}}\langle Z, \nabla U \rangle f |y|^{a}+\int_{\mathbb{S}_{r}}Uf |y|^a.
   \end{align*}
   Observing now that by \eqref{Ir1}, $D(r)=I(r)-\int_{\mathbb{B}_{r}}|y|^{a}Uf$,  we obtained the desired conclusion.
   
\end{proof}

Following \cite{GG}, we next introduce certain quantities that play a key role in the analysis of the monotonicity properties of the frequency. We consider 
\[
M(r)=\frac{H(r)}{\psi(r)},\ \ \ \ \ \ \ \ \ \ J(r)=\frac{I(r)}{\psi(r)},
\]
and define the \emph{generalised frequency} as
\[
\Phi(r)=\frac{\sigma(r)J(r)}{M(r)},
\]
where $\sigma$ is defined by \eqref{sigma}.

\begin{theorem}\label{mon1}
	Assume that $U(0)=0$. Given $\delta \in (0,1)$, there exist universal constants $r_{0}, K'>0$ such that the function $r\mapsto e^{K'r^{\frac{1-\delta}{2}}}\Phi(r)$ is non-decreasing on $\Gamma_{r_{0}}$. Precisely, for every $r\in \Gamma_{r_{0}}$ we have 
	\begin{align*}
		\frac{d}{dr}\log \Phi(r)=\frac{\Phi'(r)}{\Phi(r)}\ge - \frac{K'}{r^{\frac{1+\delta}{2}}}.
	\end{align*} 
\end{theorem}

\begin{proof}
We begin by computing the derivatives of $M(r)$ and $J(r)$. One has 
\begin{align*}
M'(r)&=-\frac{\psi'(r)H(r)}{[\psi(r)]^{2}}+\frac{H'(r)}{\psi(r)}=-\frac{\psi'(r)H(r)}{[\psi(r)]^{2}}+\frac{1}{\psi(r)}(2I(r)+\int_{\mathbb{S}_{r}}U^{2}L_a|X|)=\\
		&=\frac{1}{\psi(r)}\left[\int_{\mathbb{S}_{r}}U^{2}L_a|X|-\frac{\psi'(r)J(r)}{\psi(r)}\right]+\frac{2I(r)}{\psi(r)},
\end{align*}
where we have used Lemma \ref{Hrder}. Keeping in mind that if $r\in \Gamma_{r_{0}}$ we have $H(r)\neq 0$, by \eqref{psi} and  Definition \ref{Gdef} we have at every $r\in \Gamma_{r_{0}}$: $\frac{\psi'(r)}{\psi(r)}=G(r)=\frac{\int_{\mathbb{S}_{r}}U^{2}L_a|X|}{H(r)}$, or equivalently, $\frac{\psi'(r)}{\psi(r)}H(r)-\int_{\mathbb{S}_{r}}U^{2}L_a|X|=0$. This implies at every $r\in \Gamma_{r_{0}}$,
\[
M'(r)=\frac{2I(r)}{\psi(r)}=2J(r),\ \ \ \ \ \ \ \  \frac{M'(r)}{M(r)}=2\frac{J(r)}{M(r)}.
\]
Moreover, by \eqref{Ider1} and the fact that $I(r)=J(r)\psi(r)$, we have:
\begin{align*}
J'(r)&=\frac{1}{\psi(r)}I'(r)-\frac{\psi'(r)}{[\psi(r)]^{2}}I(r)=\frac{1}{\psi(r)}\bigg[2\int_{\mathbb{S}_{r}}\frac{(\langle A(x)\nabla U, \nu\rangle)^{2}}{\tilde{\mu}} |y|^a+\left(\frac{n-1+a}{r}+O(1)\right)I(r)\\
&+\int_{\mathbb{S}_{r}}Uf |y|^a-\left(\frac{n-1+a}{r}+O(1)\right)\int_{\mathbb{B}_{r}}Uf|y|^a-\frac{2}{r}\int_{\mathbb{B}_{r}}\langle Z,\nabla U \rangle f |y|^{a}\bigg]-\frac{\psi'(r)}{[\psi(r)]^{2}}I(r)\\
	 &=\bigg(\frac{n-1+a}{r}+O(1)\bigg)\frac{I(r)}{\psi(r)}-\frac{\psi'(r)}{\psi(r)}J(r)+\frac{1}{\psi(r)}\bigg[2\int_{\mathbb{S}_{r}}\frac{(\langle A(x)\nabla U, \nu\rangle)^{2}}{\tilde{\mu}} |y|^a\\
&+\int_{\mathbb{S}_{r}}Uf|y|^a-\left(\frac{n-1+a}{r}+O(1)\right)\int_{\mathbb{B}_{r}}Uf|y|^a-\frac{2}{r}\int_{\mathbb{B}_{r}}\langle Z,\nabla U \rangle f |y|^{a}\bigg]\\
&=\bigg(\frac{n-1+a}{r}-\frac{\psi'(r)}{\psi(r)}+O(1)\bigg)J(r)+\frac{1}{\psi(r)}\bigg[2\int_{\mathbb{S}_{r}}\frac{(\langle A(x)\nabla U, \nu\rangle)^{2}}{\tilde{\mu}}|y|^a\\
&+\int_{\mathbb{S}_{r}}Uf|y|^a -(\frac{n-1+a}{r}+O(1))\int_{\mathbb{B}_{r}}Uf|y|^a -\frac{2}{r}\int_{\mathbb{B}_{r}}\langle Z,\nabla U \rangle f |y|^{a}\bigg].
\end{align*}
We next compute $\frac{\Phi'(r)}{\Phi(r)}$.  By the definition \eqref{I}, we have:
	\begin{align*}
		\frac{\Phi'(r)}{\Phi(r)}&=\frac{\sigma'(r)}{\sigma(r)}+\frac{J'(r)}{J(r)}-\frac{M'(r)}{M(r)}=\frac{\sigma'(r)}{\sigma(r)}+\frac{J'(r)}{J(r)}-2\frac{J(r)}{M(r)}\\
		&=\frac{\sigma'(r)}{\sigma(r)}-\frac{\psi'(r)}{\psi(r)}+\frac{n-1+a}{r}+O(1)+\frac{1}{\psi(r)J(r)}\bigg[2\int_{\mathbb{S}_{r}}\frac{(\langle A(x)\nabla U, \nu\rangle)^{2}}{\tilde{\mu}}|y|^a\\
		&+\int_{\mathbb{S}_{r}}Uf |y|^a-\left(\frac{n-1+a}{r}+O(1)\right)\int_{\mathbb{B}_{r}}|y|^{a}Uf-\frac{2}{r}\int_{\mathbb{B}_{r}}\langle Z,\nabla U \rangle f |y|^{a}\bigg]-2\frac{J(r)}{M(r)}\\
		&=O(1)+\frac{1}{\psi(r)J(r)}\{2\int_{\mathbb{S}_{r}}\frac{(\langle A(x)\nabla U, \nu\rangle)^{2}}{\tilde{\mu}} |y|^a+\int_{\mathbb{S}_{r}}Uf|y|^a\\
		&-\left(\frac{n-1+a}{r}+O(1)\right)\int_{\mathbb{B}_{r}}Uf|y|^a-\frac{2}{r}\int_{\mathbb{B}_{r}}\langle Z,\nabla U \rangle f |y|^{a}\}-2\frac{J(r)}{M(r)}.
	\end{align*}
	Notice that we have: $\frac{2}{\psi(r)J(r)}\int_{\mathbb{S}_{r}}\frac{(\langle A(x)\nabla U, \nu\rangle)^{2}}{\tilde{\mu}} |y|^a-2\frac{J(r)}{M(r)}\ge 0 \iff \frac{1}{I(r)}\int_{\mathbb{S}_{r}}\frac{(\langle A(x)\nabla U, \nu\rangle)^{2}}{\tilde{\mu}} |y|^a -\frac{I(r)}{H(r)}\ge 0 \iff (I(r))^{2}\le(\int_{\mathbb{S}_{r}}\frac{(\langle A(x)\nabla U, \nu\rangle)^{2}}{\tilde{\mu}} |y|^a) H(r) $ which in turn  is true by Cauchy-Schwartz inequality. Indeed:
	\begin{align*}
		(\int_{\mathbb{S}_{r}}U\langle A(x)\nabla U, \nu \rangle |y|^{a})^{2}&=\left(\int_{\mathbb{S}_{r}}U\langle A(x)\nabla U, \nu \rangle |y|^{a/2}|y|^{a/2}\frac{\sqrt{\tilde {\mu}}}{\sqrt{\tilde {\mu}}}\right)^{2}\le \int_{\mathbb{S}_{r}}\frac{(\langle A(x)\nabla U, \nu \rangle)^{2} }{\tilde{\mu}} |y|^a \int_{\mathbb{S}_{r}}U^{2}\tilde{\mu}|y|^{a}\\
		&= H(r) \int_{\mathbb{S}_{r}}\frac{(\langle A(x)\nabla U, \nu \rangle)^{2} }{\tilde{\mu}} |y|^a,
	\end{align*}
	where we have used $\tilde\mu |y|^{a}=\mu$. Thus, it follows
	\begin{align*}
		\frac{\Phi'(r)}{\Phi(r)}\ge \frac{1}{I(r)}\left(-\frac{2}{r}\int_{\mathbb{B}_{r}}\langle Z,\nabla U \rangle f |y|^a -(\frac{n-1+a}{r}+O(1))\int_{\mathbb{B}_{r}}Uf |y|^a+\int_{\mathbb{S}_{r}}Uf |y|^a\right)+O(1).
	\end{align*}
	Now, we want to prove:
	\begin{align*}
		\frac{1}{I(r)}\left(-\frac{2}{r}\int_{\mathbb{B}_{r}}\langle Z,\nabla U \rangle f |y|^a -(\frac{n-1+a}{r}+O(1))\int_{\mathbb{B}_{r}}Uf|y|^a +\int_{\mathbb{S}_{r}}Uf |y|^a \right)\ge -K r^{-\frac{1+\delta}{2}}.
	\end{align*}
	By \eqref{rn3a1} and \eqref{rn3a2} and  since for every $r\in \Gamma_{r_{0}}$ we have $r^{1-\delta}=r^{\frac{1-\delta}{2}}r^{\frac{1-\delta}{2}}\le Cr^{\frac{1-\delta}{2}}$, therefore:
	\begin{align*}
		\left|\int_{\mathbb{B}_{r}}Uf |y|^{a}\right|\le C\big(r^{\frac{n+3+a}{2}}[D(r)]^{\frac{1}{2}}+r^{3+n+a}\big)\le C'r^{\frac{1-\delta}{2}}D(r)+Cr^{\frac{1-\delta}{2}}D(r)\le C r^{\frac{1-\delta}{2}}D(r).
	\end{align*}
	Then using \eqref{ID1} we obtain $|\int_{\mathbb{B}_{r}}Uf |y|^{a}|\le Cr^{\frac{1-\delta}{2}}I(r)$. Therefore,
	\begin{align*}
		\left|\frac{1}{rI(r)}\int_{\mathbb{B}_{r}}U f |y|^{a}\right|\le Cr^{-\frac{1+\delta}{2}} \implies -\frac{1}{I(r)}\left(\frac{n-1+a}{r}+O(1)\right)\int_{\mathbb{B}_{r}}Uf |y|^a \ge -C' r^{-\frac{1+\delta}{2}}.
	\end{align*}
	Recalling that $f\in L^{\infty}$, we have:
	\begin{align*}
		|\frac{2}{r}\int_{\mathbb{B}_{r}}\langle Z,\nabla U \rangle f |y|^{a}|&=|\frac{2}{r}\int_{\mathbb{B}_{r}} \langle \frac{A(x)X}{\tilde{\mu}},\nabla U \rangle f|y|^{a}|\le |\frac{C}{r}\int_{\mathbb{B}_{r}}|X| \langle A(x) \nabla r, \nabla U \rangle |y|^{a}|\\
		&\le \frac{C}{r}|\int_{\mathbb{B}_{r}}X(\langle A(x)\nabla r, \nabla r \rangle)^{\frac{1}{2}}(\langle A(x)\nabla U, \nabla U \rangle)^{\frac{1}{2}}|y|^{a}| \\
		&\le \frac{C}{r}\int_{B_{r}} |X| (\langle A(x)\nabla U, \nabla U \rangle)^{\frac{1}{2}}|y|^{a}|\\
		&\le \frac{C}{r} ( \int_{\mathbb{B}_{r}}\langle A(x)\nabla U, \nabla U \rangle |y|^a)^{\frac{1}{2}}(\int_{\mathbb{B}_{r}}|X|^{2}|y|^{a})^{\frac{1}{2}}\\
		&\le \frac{C}{r}(\int_{0}^{r}\rho ^{2+a+n} d\rho)^{\frac{1}{2}} [D(r)]^{\frac{1}{2}}\le C' r^{\frac{3+a+n}{2}-1}[D(r)]^{\frac{1}{2}}.
	\end{align*}
	Now by \eqref{rn3a1} we have $r^{\frac{3+a+n}{2}}\le r^{\frac{1-\delta}{2}} [D(r)]^{\frac{1}{2}}$. Thus by using \eqref{ID1}, we deduce:
	\begin{align*}
		|\frac{2}{r}\int_{\mathbb{B}_{r}}\langle Z,\nabla U \rangle f |y|^{a}|&\le C' r^{-\frac{1+\delta}{2}}[D(r)] \le C'' r^{- \frac{1+\delta}{2}}[I(r)]\\
		&\implies \frac{2}{r I(r)} \int_{\mathbb{B}_{r}}\langle Z,\nabla U \rangle f |y|^{a} \le C'' r^{- \frac{1+\delta}{2}} \\
		&\implies -\frac{2}{r I(r)} \int_{\mathbb{B}_{r}}\langle Z,\nabla U \rangle f |y|^{a} \ge -C'' r^{- \frac{1+\delta}{2}}.
	\end{align*}
	Moreover:
	\begin{align*}
		|\int_{\mathbb{S}_{r}}Uf|y|^{a}|&\le C |\int_{\mathbb{S}_{r}}U|y|^{a}|=C|\int_{\mathbb{S}_{r}}U \frac{\sqrt{\tilde{\mu}}}{\sqrt{\tilde{\mu}}}|y|^{a/2}|y|^{a/2}|\le C(\int_{\mathbb{S}_{r}}U^{2}|y|^{a}\tilde{\mu})^{\frac{1}{2}}(\int_{\mathbb{S}_{r}}\frac{|y|^{a}}{\tilde{\mu}})^{\frac{1}{2}}\\
		&=C(\int_{\mathbb{S}_{r}}U^{2}\mu)^{\frac{1}{2}}(\int_{\mathbb{S}_{r}}\frac{|y|^{a}}{\tilde{\mu}})^{\frac{1}{2}}\le C(\int_{\mathbb{S}_{r}}U^{2}\mu)^{\frac{1}{2}}r^{\frac{n+a}{2}}=C[H(r)]^{\frac{1}{2}} r^{\frac{n+a}{2}}\\\
		&\le C [D(r)]^{\frac{1}{2}}r^{\frac{n+a+1}{2}} =  C [D(r)]^{\frac{1}{2}}r^{\frac{n+a+3}{2}}r^{-1}\le [D(r)]^{\frac{1}{2}} r^{-\frac{1+\delta}{2}} [D(r)]^{\frac{1}{2}}\le C r^{-\frac{1+\delta}{2}}I(r),\\
		&\implies \frac{1}{I(r)}\int_{\mathbb{S}_{r}}Uf|y|^{a}\ge  -Cr^{-\frac{1+\delta}{2}}.
	\end{align*}
	Note that in the above inequality, we also  used \eqref{estHr1} and \eqref{ID1}. \\
	Thus we have proved that for $r \in \Gamma_{r_0}$, 
	\begin{align*}
		T(r):=\frac{1}{I(r)}\left(-\frac{2}{r}\int_{\mathbb{B}_{r}}|y|^{a}\langle Z,\nabla U \rangle f-(\frac{n-1+a}{r}+O(1))\int_{\mathbb{B}_{r}}|y|^{a}Uf+\int_{\mathbb{S}_{r}}|y|^{a}Uf\right)\ge -K r^{-\frac{1+\delta}{2}},
	\end{align*} 
	which implies that  $\frac{\Phi'(r)}{\Phi(r)}\ge -K' r^{-\frac{1+\delta}{2}}$ since
	 $|\frac{\Phi'(r)}{\Phi(r)}-T(r)|\le C$ for all  $r$  small enough. This concludes the proof.
	 
\end{proof}

With Theorem \ref{mon1} in hands, by an analogous argument as in the proof of Theorem  5.12 in \cite{GG}, we obtain the following monotonicity result.

\begin{theorem}\label{amon}
Assume $U(0)=0$. With $r_0, K'$ as  in Theorem \ref{mon1} corresponding to some choice of $\delta \in (0,1)$, we have that
\[
r \to N(r) \overset{\rm def}{=} \frac{\sigma(r)}{2} e^{K' r^{\frac{1-\delta}{2}}} \frac{d}{dr} \log \  \max ( M(r), r^{3+\delta})
\]
is non-decreasing in $(0, r_0)$. In particular $N(0+)$ exists. 

\end{theorem}
We also need to work with the following quantity
\begin{equation}\label{tiln1}
\tilde N(r) \overset{\rm def}{=} \frac{r}{\sigma(r)} N(r).
\end{equation}
Now it follows from Lemma \ref{ese} and Theorem \ref{amon} that the following holds.
\begin{corollary}\label{tiln}
Let $\tilde N(r)$ be defined as in \eqref{tiln1}.   Then $\tilde N(0+)$ exists.

\end{corollary}

\section{Optimal regularity}\label{optregu}
We now choose $\delta$  in Theorem \ref{amon} such that $3+ \delta > 3-a$. Our next result concerns the optimal decay  of $U$ near a free boundary point. 
\begin{theorem}\label{opt}
Let $U$ be a solution to \eqref{La1} and let $X_0=(x_0, 0)  \in \Gamma(U)$. Then we have that 
\begin{equation}\label{optd}
|U(X)| \leq C |X-X_0|^{\frac{3-a}{2}}
\end{equation}
for some universal constant $C$.

\end{theorem}
\begin{proof}
Without loss of generality we may assume that $X_0=(0,0)$ and it suffices to show that
\begin{equation}\label{suf}
||U||_{L^{\infty}(\B_r^+)} \leq C r^{\frac{3-a}{2}}.
\end{equation}
By rotation of coordinates, we may assume that $A(0)=\mathbb{I}$. Let $d_r= M(r)^{\frac{1}{2}}$ and  consider the following Almgren type rescalings  $U_r(X)= \frac{U(rX)}{d_r}$. Note that $U_r$ solves the Signorini problem in \eqref{La1} corresponding to $A_r(X)= A(rX)$ and $f_r= r^2 \frac{f(rX)}{d_r}$ and $0 \in \Gamma(U_r)$.  We note that the Lipschitz norm of $A_r$ is bounded from above by the Lipschitz norm of $A$. Now given the validity of Lemma \ref{ese} as well as  the monotonicity result in Theorem \ref{amon}, by an analogous blowup argument ( which uses Theorem \ref{w2}, Theorem \ref{holdery} and Theorem \ref{alp} ) as in the proof of Lemma 6.3 in \cite{GG} for $a=0$ ( see also the proof  of Lemma  6.2 in \cite{CSS} for $a \in (-1, 1)$ and $A=\mathbb{I}$) we obtain that  $\tilde N(0+) \geq \frac{3-a}{2}$ with $\tilde N$ as in \eqref{tiln1}.  We note that this crucially utilizes the fact that up to a subsequence,  $U_r \to U_0$ which is a homogeneous solution to the Signorini problem in \eqref{La1} with $A=\mathbb{I}$ and $f=0$ and that the homogeneity of $U_0 \geq  \frac{3-a}{2}$, thanks to Theorem 5.7 in \cite{CSS}.  Then by using the monotonicity result  in Theorem \ref{amon} and  by arguing as in the proof of Lemma 6.4 in \cite{GG} we obtain that 
\begin{equation}\label{opt1}
H(r) \leq C r^{n+3}.
\end{equation}
for $r \in (0, r_0)$ with $r_0$ as in Theorem \ref{amon} above. From \eqref{opt1} it follows that
\begin{equation}\label{opt2}
\int_{\B_r}  |y|^a U^2 \leq  Cr ^{n+4}.
\end{equation}
 Now we  note that $U^+$ and $U^-$ are subsolutions to
\begin{equation}\label{su}
L_a v \geq - |y|^a C
\end{equation}
where $C= ||f||_{L^{\infty}}$. This is seen by arguing as in Lemma 2.5 in \cite{GG}. Then we note that  from \eqref{su}, it follows that   $w= U^{+} +  \frac{C}{2(1+a)} y^2$ solves
\begin{equation}\label{su1}
L_a w \geq 0.
\end{equation}
Moreover using \eqref{opt2} it is seen that 
\begin{equation}\label{opt4}
\int_{\B_r} |y|^a w^2 \leq  Cr ^{n+4}.
\end{equation}
Thus from the subsolution estimates as in \cite{FKS}, we have that
\[
\text{sup}_{\B_{r/2}} w \leq C r^{\frac{3-a}{2}}
\]
from which we obtain that
\[
\text{sup}_{\B_{r/2}} U^{+} \leq Cr^{\frac{3-a}{2}}.
\]
And analogous  argument holds for $U^{-}$ and we  thus  conclude that \eqref{suf} holds.

\end{proof}
We also note that the following gap of frequency follows from Theorem 5.7 in \cite{CSS} which concerns the degree of homogeneous global solutions to the constant coefficient Signorini problem (i.e. $A=\mathbb{I}$)  with  zero obstacle.

\begin{lemma}\label{gap}
Let $0 \in \Gamma(U)$ and assume that $A(0)=\mathbb{I}$. Then either $\tilde N(0+) =\frac{3-a}{2}$ or $\tilde N(0+) \geq \frac{3+\delta}{2}$. 
\end{lemma}
We now proceed with the proof of optimal regularity as stated in Theorem \ref{optreg}.

\begin{proof}[Proof of Theorem \ref{optreg}]
 The  proof is similar to that  of Theorem \ref{alp} in view of the improved decay estimate in \eqref{optd}. We nevertheless provide the complete  details.     By subtracting off the obstacle, we may assume that $U$ solves \eqref{La1} with $f$ independent of $y$.  Given $X \in \mathbb{B}_{\frac{1}{2}}^+$, let $d(X)= d(X, \Gamma(U))$.  We note that in $\mathbb{B}_{d(X)}(X) \cap \{y=0\}$, either $\p_y^a U$ or $U$ identically vanishes. Therefore by even or odd reflection, we have that $U$ solves in $\mathbb{B}_{d(X)}(X),$ \[
 \operatorname{div}(|y|^a A(x) \nabla U)= |y|^a f
 \]
 and moreover from \eqref{optd} we have,
 \begin{equation}\label{k20}
 ||U||_{L^{\infty}(\mathbb{B}_{d}(X))} \leq C d(X)^{\frac{3-a}{2} }.
 \end{equation}

Then from the estimate \eqref{k2} it follows by using scaled versions of the estimates in Theorem \ref{even1} or \ref{odd1} that the following gradient bounds holds, 

 \begin{equation}\label{m1}
 |\nabla_x U( X)| \leq C d(X)^{\frac{1-a}{2}}.
 \end{equation}
 
 We now take points $X^1, X^2$ and let $d_i= d(X^i, \Gamma(U))$ for $i=1,2$. Without loss of generality assume that $d_1 \geq d_2$. We also set $\delta= |(X^1 - X^{2}|$. . There exist two possibilities: (a) $\delta \ge \frac{1}{8} d_1$; or, (b) $\delta < \frac{1}{8} d_1$.
If (a) occurs, it follows from \eqref{m1} that
\[
\begin{aligned}
|\nabla_x U(X^1) - \nabla_x U(X^2) | & \leq |\nabla_x U(X^1)| + |\nabla_x U(X^2)| \\
 &\leq C d_1^{\frac{1-a}{2}} + C d_2^{\frac{1-a}{2}} \leq C \delta^{\frac{1-a}{2}}.
\end{aligned}
\]
If (b) occurs, then we have that $X^2 \in \mathbb{B}_{\frac{d_1}{8}}(X^1)$.  it follows from the rescaled  estimates in Theorem \ref{even1} or \ref{odd1} ( corresponding to $\beta=\frac{1-a}{2}$) that the following holds, 
\[
\begin{aligned}
|\nabla_x U(X^1) - \nabla_x U(X^2) | & \leq \frac{C}{d_1^{\frac{3-a}{2}}} ( ||U||_{L^{\infty}(\mathbb{B}_{d_1}(X_1))} + d_1^2 ||f||_{L^{\infty}})  ) \delta^{\frac{1-a}{2}}\\
& \leq C \delta^{\frac{1-a}{2}},
\end{aligned}
\]
where we also used the decay estimate in \eqref{k20} above. Thus in both cases, we obtain,
\[
|\nabla_x U(X^1) - \nabla_x U(X^2) | \leq C |X^1- X^2|^{\frac{1-a}{2}}.
\]
We now establish the optimal H\"older regularity for $y^a U_y$. Again given $X \in \B_{\frac{1}{2}}^+$,  we note that  either  $U \equiv 0$ or $\p_y^a U \equiv 0$ on $\B_{d(X)} (X) \cap \{y=0\}$. If   $y^a U_y \equiv 0$, then $w= y^a U_y$ can be oddly reflected across $\{y=0\}$ so that it solves 
\begin{equation}\label{od10}
\var(|y|^{-a} A \nabla w) =0
\end{equation}
in $\B_{d(X)}(X)$.   We then claim that the  following decay estimate holds for $y^a U_y$ near a free boundary point,
\begin{equation}\label{ne1}
|y^a U_y (X) | \leq C d(X)^{\frac{1+a}{2}}.
\end{equation}
The estimate \eqref{ne1}  is a consequence of the  following Moser type estimate  as in \cite{FKS},
\[
|y^a U_y (X) | \leq \frac{C}{d(X)^{\frac{n+1-a}{2}}} \left( \int_{\B_{d(X)/4}(X) } |y|^{-a} (|y|^a U_y)^2 \right)^{\frac{1}{2}},
\]
combined  with the  following energy estimate, 
\[
\int_{\B_{d(X)/4}(X) } |y|^{-a} (|y|^a U_y)^2  \leq \frac{C}{d(X)^2} \int_{\B_{d(X)/2}(X) } |y|^{a} U^2
\]
and the decay estimate for $U$ as in \eqref{k20}.  If instead $U\equiv 0$ in  $\B_{d(X)}(X) \cap \{y=0\}$, then we can  use the estimate \eqref{even4} and the decay for $U$ in \eqref{k20}  to again deduce that \eqref{ne1} holds. 

We now establish the $C^{\frac{1+a}{2}}$ regularity of $y^a U_y$ in $\overline{\B_{\frac{1}{2}}^+}$. Again we  take points $X^1, X^2$ and let $d_i= d(X^i, \Gamma(U))$ for $i=1,2$. Without loss of generality assume that $d_1 \geq d_2$. We also set $\delta= |(X^1 - X^{2}|$. . There exist two possibilities: (a) $\delta \ge \frac{1}{8} d_1$; or, (b) $\delta < \frac{1}{8} d_1$.
If (a) occurs,  the using the decay estimate in \eqref{ne1} we obtain,
\[
|y^a U_y(X^1) - y^a U_y(X^2)| \leq |y^a U_y(X^1)| + |y^aU_y(X^2) | \leq C \left( d_1^{\frac{1+a}{2}}+ d_2^{\frac{1+a}{2}} \right) \leq C \delta^{\frac{1+a}{2}}.
\]
If instead (b) occurs, then we have that $X^2 \in \B_{\frac{d_1}{8}}(X^1)$.  Then as before, we note  that either $U$ or $\p_y^a U$ vanishes identically on $\B_{d_1}(X^1) \cap \{y=0\}$. If $U \equiv 0$ on $B_{d_1}(X^1) \cap \{y=0\}$, then from the scaled version of the  estimate \eqref{even4} in Theorem \ref{odd1} ( corresponding to $\alpha= \frac{1+a}{2}$) we have that,
\begin{equation}\label{holl1}
|y^a U_y(X^1) - y^a U_y(X^2) | \leq \frac{C}{d_1^{\frac{3-a}{2}}} \left( ||U||_{L^{\infty}(\B_{d_1}(X^1))}  + d_1^{2} ||f||_{L^{\infty}} \right) \delta^{\frac{1+a}{2}} \leq C \delta^{\frac{1+a}{2}},
\end{equation}
where in the last inequality in \eqref{holl1} above, we used the decay estimate \eqref{k20} for $X=X^1$.  On the other hand, if $\p_y^a U \equiv 0$ on $\B_{d_1}(X^1) \cap \{y=0\}$, then we can extend $y^a U_y$  in an odd way across $\{y=0\}$  so that it is an odd in $y$ solution to \eqref{od10} in $\B_{d_1}(X^1)$.  Now from the rescaled estimate in  Theorem \ref{odd2}  (corresponding to  $\alpha= \frac{1+a}{2}$) we get,
\begin{equation}\label{holl2}
|y^a U_y(X^1) - y^aU_y(X^2) | \leq \frac{C}{d_1^{\frac{1+a}{2}} }||y^a U_y||_{L^{\infty}(\B_{d_1}(X^1))} \delta^{\frac{1+a}{2}} \leq C \delta^{\frac{1+a}{2}},
\end{equation}
where in the last inequality in \eqref{holl2} above, we used the estimate \eqref{ne1} for $X=X^1$. The conclusion thus follows.

\end{proof}

\begin{remark}
We note that it is not true that the solution is $C^{1,s}$ in the $y$ variable. See for instance Remark 4.5 in \cite{CSS} for further discussion on this aspect. 
\end{remark}

\section{Smoothness of the regular set of the free boundary}\label{fbreg}
We now define the notion of regular points to \eqref{La}. Let $(x_0, 0) \in \Gamma(U)$. Let $U_{x_0}(x, y)= U(x_0+ A^{\frac{1}{2}} (x_0)x, y)$,  $A_{x_0}(x,y)= A^{-\frac{1}{2}}(x_0) A(x_0 + A^{\frac{1}{2}}(x_0) x) A^{-\frac{1}{2}}(x_0)$. Under this normalization, we have that  $U_{x_0}$ solves \eqref{La} corresponding to the new matrix $A_{x_0}$ and moreover we have that $0 \in \Gamma(U_{x_0})$ and $A_{x_0}(0)=\mathbb{I}$.  Again by subtracting off the obstacle, we have that $U_{x_0}$ solves a problem of the type \eqref{La1}. We thus have $N, \tilde N$  have limits at $0$ defined with respect to the new operator $A_{x_0}$ and for notational convenience, we denote such quantities by $N_{x_0}, \tilde N_{x_0}$ etc.

\begin{definition}\label{D:reg}
Let $U$ be a solution of \eqref{La1}. We say that $0\in \Gamma(U)$ is a \emph{regular free boundary point} if $
\tilde N(0+) = \frac{3-a}{2}$. Likewise,  we say that $X_0 = (x_0,0)$ is regular if $\tilde N_{x_0}(0+)= \frac{3-a}{2}$.  We denote by $\Gamma_{\frac{3-a}{2}}(U)$ the set of all regular free
boundary points and we call it the \emph{regular set} of $U$.
\end{definition}

For the analysis of the regular set we will need the following result which generalises \cite[Theorem 4.3]{GPG}.

\begin{theorem}[Weiss type monotonicity formula]\label{wmon}
Given a solution $U$ to \eqref{La1}, such that $0\in \Gamma_{\frac{3-a}{2}}(U)$, define
\begin{equation}\label{W}
W(U,r)=W(r)=\frac{\sigma(r)}{r^{3-a}}\{J(r)-\frac{3-a}{2r}M(r)\}.
\end{equation}
There exist universal constants $C,r_{0}>0$, depending on $||f||_{L^{\infty}(B_{1})}$, such that for any $0<r<r_{0}$ one has:
\begin{align*}
\frac{d}{dr} (W(U,r)+Cr^{\frac{1+a}{2}})&\ge \frac{2}{r^{n+2}}\int_{\mathbb{S}_{r}}\left(\frac{\langle A \nabla U, \nu \rangle}{\sqrt{\tilde{\mu}}}-\frac{3-a}{2r}\sqrt{\tilde{\mu}}U\right)^{2} |y|^a\\
&=\frac{2}{r^{n+2}}\int_{\mathbb{S}_{r}}\left(\frac{\langle A \nabla U, \nu \rangle}{\sqrt{\tilde{\mu}}}-\frac{3-a}{2r}\sqrt{\tilde{\mu}}U\right)^{2} |y|^a.
\end{align*}
	In particular, there exists $C>0$ such that the function $r\mapsto W(U,r)+Cr^{\frac{1+a}{2}}$ is monotone increasing and therefore the limit $W(U,0^{+}):=\lim_{r\to 0^+}W(U,r)$ exists.
\end{theorem}

\begin{proof}
Differentiating \eqref{W} we find
\begin{align*}
\frac{d}{dr}W(r)&=\frac{\sigma(r)}{r^{3-a}}\left(J'(r)+\frac{3-a}{2r^{2}}M(r)-\frac{3-a}{2r}M'(r)\right)
\\
& +\left(\frac{\sigma'(r)}{r^{3-a}}-(3-a)\frac{\sigma(r)}{r^{4-a}}\right)\left(J(r)-\frac{3-a}{2r}M(r)\right)\\
   	&=\frac{\sigma(r)}{r^{3-a}}\bigg[\left(\frac{\sigma'(r)}{\sigma(r)}-\frac{3-a}{r}\right)\left(J(r)-\frac{3-a}{2r}M(r)\right)+\left(J'(r)+\frac{3-a}{2r^{2}}M(r)-\frac{3-a}{2r}M'(r)\right)\bigg].
   \end{align*}
    After some easy computations, by  recalling the expression we found for $J'(r)$ in  the proof of Theorem \ref{mon1}, we get
   \begin{align*}
   	& \frac{d}{dr}W(r)=\frac{\sigma(r)}{r^{3-a}}\bigg[(\frac{\psi'(r)}{\psi(r)}-\frac{n-1+a}{r}-\frac{3-a}{r})(J(r)-\frac{3-a}{2r}M(r))+(J'(r)+\frac{3-a}{2 r^{2}}M(r)-\frac{3-a}{2r}M'(r))\bigg]\\
   	 &=\frac{\sigma(r)}{r^{3-a}}\bigg[\left(\frac{\psi'(r)}{\psi(r)}-\frac{n-1+a}{r}-\frac{3-a}{r}\right)\left(J(r)-\frac{3-a}{2r}M(r)\right)+\left(\frac{n-1+a}{r}-\frac{\psi'(r)}{\psi(r)}+O(1)\right)J(r)\\
   	 &+\frac{1}{\psi(r)}\bigg(2\int_{\mathbb{S}_{r}}\frac{(\langle A \nabla U, \nu \rangle)^{2}}{\tilde {\mu}} |y|^a +\int_{\mathbb{S}_{r}}Uf|y|^a-(\frac{n-1+a}{r}+O(1))\int_{\mathbb{B}_{r}}Uf|y|^a-\frac{2}{r}\int_{\mathbb{B}_{r}}\langle Z, \nabla U \rangle f |y|^{a}\bigg)\\
   	 &-\frac{3-a}{2r}M'(r)+\frac{3-a}{2r^{2}}M(r) \bigg].\notag
   \end{align*}
   Now from the  proof  of Theorem \ref{mon1}  we have  that $M'(r)=2J(r)$ and hence using this we obtain,
   \begin{align*}
   	\frac{d}{dr}W(r)&=\frac{\sigma(r)}{r^{3-a}}\bigg[\bigg(\frac{\psi'(r)}{\psi(r)}-\frac{n-1+a}{r}-\frac{3-a}{r}+\frac{n-1+a}{r}-\frac{\psi'(r)}{\psi(r)}+O(1)-\frac{3-a}{r}\bigg)J(r)\\
   	&+\frac{2}{\psi(r)}\int_{\mathbb{S}_{r}}\frac{(\langle A \nabla U, \nu \rangle)^{2}}{\tilde{\mu}} |y|^a+\bigg(\frac{3-a}{2r^{2}}-\frac{3-a}{2r}(\frac{\psi'(r)}{\psi(r)}-\frac{n-1+a}{r}-\frac{3-a}{r})\bigg)M(r)\\
   	&-\frac{1}{\psi(r)}\bigg(\frac{n-1+a}{r}+O(1)\bigg)\int_{\mathbb{B}_{r}}Uf|y|^a-\frac{2}{r\psi(r)}\int_{\mathbb{B}_{r}}\langle Z, \nabla U \rangle f |y|^{a}+\frac{1}{\psi(r)}\int_{\mathbb{S}_{r}}Uf|y|^a\bigg].
   	\end{align*}
	Proceeding further, we get,
   	\begin{align*}
   	\frac{d}{dr}W(r)&=\frac{\sigma(r)}{r^{3-a}}\bigg[2\left(-\frac{3-a}{r}+O(1)\right)J(r)+\frac{2}{\psi(r)}\int_{\mathbb{S}_{r}}\frac{(\langle A \nabla U, \nu \rangle )^{2}}{\tilde{\mu}} |y|^a\\
   	&+\frac{3-a}{2r^{2}}\bigg(1-(r\frac{\psi'(r)}{\psi(r)}-(n-1+a)-3+a) \bigg)M(r)\\
   	&-\frac{1}{\psi(r)}\bigg(\left(\frac{n-1+a}{r}+O(1)\right)\int_{\mathbb{B}_{r}}Uf|y|^a+\frac{2}{r}\int_{\mathbb{B}_{r}}\langle Z, \nabla U \rangle f |y|^{a}-\int_{\mathbb{S}_{r}}Uf|y|^a\bigg) \bigg]\\   	
   	&=\frac{\sigma(r)}{r^{3-a}}\big[2\left(-\frac{3-a}{r}+O(1)\right)J(r)+\frac{2}{\psi(r)}\int_{\mathbb{S}_{r}}\frac{(\langle A \nabla U, \nu \rangle )^{2}}{\tilde{\mu}}|y|^a +\frac{3-a}{2r^{2}}(-r\frac{\psi'(r)}{\psi(r)}+3+n)M(r)\\
   	&-\frac{1}{\psi(r)}\bigg(\frac{n-1+a}{r}+O(1))\int_{\mathbb{B}_{r}}Uf|y|^a+\frac{2}{r}\int_{\mathbb{B}_{r}}\langle Z, \nabla U \rangle f |y|^{a}-\int_{\mathbb{S}_{r}}Uf|y|^a\bigg) \bigg].\\
   \end{align*}
   Now from \eqref{psi} and Lemma \ref{Glem}, we observe that  $\frac{\psi'(r)}{\psi(r)}=\frac{n+a}{r}+O(1) \implies r\frac{\psi'(r)}{\psi(r)}=n+a+O(r)$. Subsequently by  recalling  the definitions of $J(r)$ and $M(r)$ we obtain,
   \begin{align*}
   	\frac{d}{dr}W(r)&=\frac{\sigma(r)}{r^{3-a}}\bigg[2\left(-\frac{3-a}{r}+O(1)\right)J(r)+\frac{2}{\psi(r)}\int_{\mathbb{S}_{r}}\frac{(\langle A \nabla U, \nu \rangle )^{2}}{\tilde{\mu}}|y|^a+\frac{3-a}{2r^{2}}(-n-a+3+n+O(r))M(r)\\
   	&-\frac{1}{\psi(r)}\bigg( \bigg(\frac{n-1+a}{r}+O(1)\bigg)\int_{\mathbb{B}_{r}}Uf |y|^a+\frac{2}{r}\int_{\mathbb{B}_{r}}\langle Z, \nabla U \rangle f |y|^{a}-\int_{\mathbb{S}_{r}}Uf|y|^a\bigg) \bigg]\\
   	&=\frac{2\sigma(r)}{r^{3-a}\psi(r)}\bigg[\left(\frac{-3+a}{r}+O(1)\right)I(r)+\int_{\mathbb{S}_{r}}\frac{(\langle A \nabla U,\nu \rangle)^{2}}{\tilde{\mu}}|y|^a+\left(\frac{3-a}{2r}\right)^2(1+O(r))H(r)\bigg]\\
  	&+\frac{\sigma(r)}{r^{3-a}\psi(r)}\bigg[-\left(\frac{n-1+a}{r}+O(1)\right)\int_{\mathbb{B}_{r}}Uf|y|^a-\frac{2}{r}\int_{\mathbb{B}_{r}}\langle Z,\nabla U \rangle f |y|^{a}+\int_{\mathbb{S}_{r}}Uf|y|^a\bigg].
   \end{align*}
   We also have,
   \begin{align*}
   	\int_{\mathbb{S}_{r}}\bigg(\frac{\langle A \nabla U, \nu \rangle}{\sqrt{\tilde{\mu}}}-\frac{3-a}{2r}U\sqrt{\tilde{\mu}}\bigg)^{2} |y|^a&=\int_{S_{r}}\frac{(\langle A \nabla U,\nu \rangle)^{2}}{\tilde{\mu}} |y|^a+\left(\frac{3-a}{2r}\right)^{2}\int_{\mathbb{S}_{r}}U^{2}\tilde{\mu}|y|^a\\
   	&-\frac{3-a}{r}\int_{\mathbb{S}_{r}}U\langle A \nabla U ,\nu \rangle |y|^a\\
   	&=\frac{-3+a}{r}I(r)+\int_{\mathbb{S}_{r}}\frac{(\langle A \nabla U,\nu \rangle)^{2}}{\tilde{\mu}} |y|^a+\left(\frac{3-a}{2r}\right)^{2}H(r).
   \end{align*}
   Thus
   \begin{align*}
   	\frac{d}{dr}W(r)&=\frac{2\sigma(r)}{r^{3-a}\psi(r)}\bigg[\int_{\mathbb{S}_{r}}\bigg(\frac{\langle A \nabla U,\nu \rangle}{\sqrt{\tilde{\mu}}}-\frac{3-a}{2r}\sqrt{\tilde{\mu}}U\bigg)^{2}|y|^a+O(1)I(r)+\frac{O(1)}{r}H(r)\bigg]\\
   	&+\frac{\sigma(r)}{r^{3-a}\psi(r)}\bigg[-\bigg(\frac{n-1+a}{r}+O(1)\bigg)\int_{\mathbb{B}_{r}}Uf|y|^a-\frac{2}{r}\int_{\mathbb{B}_{r}}\langle Z,\nabla U \rangle f |y|^{a}+\int_{\mathbb{S}_{r}}Uf|y|^a\bigg].
   \end{align*}
Now since  $\frac{\sigma(r)}{\psi(r)}=\frac{1}{r^{n-1+a}}$, therefore we have that  $\frac{\sigma(r)}{r^{3-a}\psi(r)}=\frac{1}{r^{2+n}}$.  Using this we obtain,
   \begin{align*}
   \frac{d}{dr}W(r)&=\frac{2}{r^{2+n}}\bigg[\int_{\mathbb{S}_{r}}\bigg(\frac{\langle A \nabla U,\nu \rangle}{\sqrt{\tilde{\mu}}}-\frac{3-a}{2r}\sqrt{\tilde{\mu}}U\bigg)^{2}|y|^a+O(1)I(r)+\frac{O(1)}{r}H(r)\bigg]\\
   &+\frac{1}{r^{2+n}}\bigg[-\bigg(\frac{n-1+a}{r}+O(1)\bigg)\int_{\mathbb{B}_{r}}Uf|y|^a-\frac{2}{r}\int_{\mathbb{B}_{r}}\langle Z,\nabla U \rangle f |y|^{a}+\int_{\mathbb{S}_{r}}Uf|y|^a\bigg].
   \end{align*}
   Now by applying the Cauchy-Schwartz inequality and also by using Theorem \ref{opt} we have,
   \begin{align*}
   	|I(r)|&\le \int_{\mathbb{S}_{r}}|U||\langle A \nabla U,\nu \rangle||y|^a \le \big(\int_{\mathbb{S}_{r}} U^{2} |y|^a\big)^{\frac{1}{2}}\big(\int_{\mathbb{S}_{r}}(\langle A \nabla U,\nu \rangle )^{2}|y|^a\big)^{\frac{1}{2}}\le \\
   	&\le C r^{(n+a)/2}r^{\frac{3-a}{2}}\big(\int_{\mathbb{S}_{r}}(\langle A \nabla U,\nu \rangle)^{2}|y|^a\big)^{\frac{1}{2}}.
   \end{align*}
  Now again since $0 \in \Gamma(U)$ which  in particular implies that  $y^aU_y(0)=\nabla_x U(0)=0$, therefore using Theorem \ref{optreg} we infer that the following estimate holds, 
   \begin{align*}
   	|y|^{a}\langle A \nabla U,\nu \rangle \le C |y|^{a}|\nabla U|\le C(|y|^{a}|\nabla_{x}U|+|y|^{a}|\partial_{y}U|)\le Cr^{\frac{1+a}{2}}.
   \end{align*} 
   Thus $O(1)I(r)=O(r^{n+2})$. Also by \eqref{opt1} we have  $\frac{O(1)H(r)}{r}=O(r^{n+2})$ and hence we obtain,
   \begin{align*}
   \frac{d}{dr}W(r)&=\frac{2}{r^{2+n}}\int_{\mathbb{S}_{r}}\bigg(\frac{\langle A \nabla U,\nu \rangle}{\sqrt{\tilde{\mu}}}-\frac{3-a}{2r}\sqrt{\tilde{\mu}}U\bigg)^{2} |y|^a +O(1)\\
   &+\frac{1}{r^{2+n}}\bigg[-\bigg(\frac{n-1+a}{r}+O(1)\bigg)\int_{\mathbb{B}_{r}}Uf|y|^a -\frac{2}{r}\int_{\mathbb{B}_{r}}\langle Z,\nabla U \rangle f |y|^{a}+\int_{\mathbb{S}_{r}}Uf|y|^a\bigg].
   \end{align*}
   Again by  Theorem \ref{optreg}, we have,
   \begin{align*}
   	\bigg|-\frac{2}{r}\int_{\mathbb{B}_{r}}\langle Z,\nabla U \rangle f |y|^{a}\bigg|\le \frac{2}{r}\int_{\mathbb{B}_{r}}|Z|\nabla U| |f| |y|^{a}\le \frac{C}{r}r r^{\frac{1+a}{2}}r^{n+1}=Cr^{n+1+\frac{1}{2}+\frac{a}{2}}
   \end{align*}
   and by Theorem \ref{opt} we also have,
   \begin{align*}
   	\bigg|-\bigg(\frac{n-1+a}{r}+O(1)\bigg)\int_{\mathbb{B}_{r}}|y|^{a}Uf+\int_{\mathbb{S}_{r}}|y|^{a}Uf\bigg|\le C r^{n+1+\frac{a+1}{2}}.
   	\end{align*}
   Now since $\frac{a-1}{2}<0$, thus  we have that  $O(1)\ge -C' \ge -C' r^{\frac{a-1}{2}}$ and so   we  finally obtain,
   \begin{align*}
   	\frac{d}{dr}W(r)&\ge\frac{2}{r^{2+n}}\int_{\mathbb{S}_{r}}\big(\frac{\langle A \nabla U,\nu \rangle}{\sqrt{\tilde{\mu}}}-\frac{3-a}{2r}\sqrt{\tilde{\mu}}U\big)^{2} |y|^a -Cr^{\frac{a-1}{2}}
   \end{align*}
   which concludes the proof.
\end{proof}

\begin{proof}[Proof of Theorem \ref{smooth}]
Now given the Weiss type monotonicity as in Theorem \ref{wmon}, together with the  epiperimetric inequality   established in \cite{GPPS} (see Theorem 4.2 in \cite{GPPS}), we can argue as in  \cite{GPG} and \cite{GPPS} to conclude that locally $\Gamma_{\frac{3-a}{2}}$  is a $C^{1+\gamma}$ graph for some $\gamma>0$.

\end{proof}




\begin{thebibliography}{00}


\bibitem{AU}
A. A. Arkhipova and N. N. Uraltseva, \emph{Regularity of the solution of a
  problem with a two-sided limit on a boundary for elliptic and
  parabolic equations}, Trudy Mat. Inst. Steklov. \textbf{179} (1988),
5--22, 241 (Russian). Translated in Proc. Steklov Inst. Math. 1989, no.~2,  Boundary value problems of mathematical physics, 13.

\bibitem{AC}
I. Athanasopoulos \& L. A. Caffarelli, \emph{Optimal regularity of lower dimensional obstacle problems}. Zap. Nauchn. Sem. S.-Peterburg. Otdel. Mat. Inst. Steklov. (POMI) \textbf{310}~(2004), Kraev. Zadachi Mat. Fiz. i Smezh. Vopr. Teor. Funkts. 35 [34], 49-66, 226; reprinted in J. Math. Sci. (N.Y.) \textbf{132}~(2006), no. 3, 274-284. 

\bibitem{ACM}
I. Athanasopoulos, L. Caffarelli \& E. Milakis, \emph{On the
  regularity of the non-dynamic parabolic fractional obstacle
  problem},  J. Differential Equations \textbf{265}~(2018), no.~6, 2614--2647.







\bibitem{ACS}
I. Athanasopoulos, L. A. Caffarelli \& S. Salsa, \emph{The structure of the free boundary for lower dimensional obstacle problems}. Amer. J. Math. \textbf{130}~(2008), no. 2, 485-498. 

\bibitem{BDGP1}
A. Banerjee, D. Danielli, N. Garofalo \& A. Petrosyan, \emph{The regular free boundary in the thin obstacle problem for  degenerate parabolic equations},  Algebra i Analiz \textbf{32}~ (2020), no. 3, 84-126. (ArXiv:1906.06885)

\bibitem{BDGP2}
A. Banerjee, D. Danielli, N. Garofalo \& A. Petrosyan, \emph{The structure of the singular set in the thin obstacle problem for degenerate parabolic equations}. Calc. Var. Partial Differential Equations \textbf{60}~(2021), no. 3, Paper no. 91, 52 pp.


\bibitem{BG}
A. Banerjee \& N. Garofalo, \emph{Monotonicity of generalized frequencies and the strong unique continuation property for fractional parabolic equations}. Adv. Math. \textbf{336}~(2018), 149-241.

\bibitem{C}
L. A. Caffarelli, \emph{Further regularity for the Signorini problem}. Comm. Partial Differential Equations \textbf{4}(9)~(1979), 1067-1075.

\bibitem{CDS}
L. A. Caffarelli, D. De Silva \& O. Savin,  \emph{The two membranes problem for different operators}. Ann. Inst. H. Poincar\'e Anal. Non Lin\'eaire \textbf{34}~(2017), no.~4, 899--932.

\bibitem{CSS}
L. A. Caffarelli, S. Salsa \& L. Silvestre,  \emph{Regularity estimates for the solution and the free boundary of the obstacle problem for the fractional Laplacian}, Invent. Math. \textbf{171}~(2008), no. 2, 425-461.



\bibitem{CS}
L. Caffarelli \& L. Silvestre, \emph{An extension problem related to the fractional Laplacian}.
Comm. Partial Diff. Eq. \textbf{32}~ (2007), no. 7-9, 1245-1260.

\bibitem{CSt}
L. Caffarelli \& P. Stinga, \emph{Fractional elliptic equations, Caccioppoli estimates and regularity},  Ann. Inst. H. Poincar\'e Anal. Non Line\'aire \textbf{33}~ (2016), no. 3, 767-807.


\bibitem{Ch}
S. Chua, \emph{Extension theorems on weighted Sobolev spaces}.
Indiana Univ. Math. J. \textbf{41}~ (1992), no. 4, 1027-1076.

\bibitem{DGPT}
D. Danielli, N. Garofalo, A. Petrosyan \& T. To, \emph{Optimal regularity and the free boundary in the parabolic Signorini problem}, Mem. Amer. Math. Soc. \textbf{249}~(2017), no.~1181, v + 103 pp. 

\bibitem{DS}
D. De Silva \& O. Savin, \emph{Boundary Harnack estimates in slit domains and applications to thin free boundary problems}, Rev. Mat. Iberoam. \textbf{32}~ (2016), no. 3, 891-912. 


\bibitem{FKS}
 E. Fabes, C. Kenig \& R. Serapioni, \emph{The local regularity of solutions of degenerate elliptic equations}, Comm. Partial Differential Equations, \textbf{7}~(1982), 77-116. 
 
\bibitem{FS}
M. Focardi \& E. Spadaro, \emph{An epiperimetric inequality for the thin obstacle problem}. Adv. Differential Equations \textbf{21}~(2016), no. 1-2, 153-200. 

\bibitem{FS2}
M. Focardi \& E. Spadaro, \emph{On the measure and the structure of the free boundary of the lower dimensional obstacle problem}. Arch. Ration. Mech. Anal. \textbf{230}~(2018), no. 1, 125-184. 
 
 \bibitem{G}
N. Garofalo, \emph{Fractional thoughts}, New developments in the analysis of nonlocal operators, 1-135, Contemp. Math., \textbf{723}, Amer. Math. Soc., Providence, RI, 2019.

\bibitem{GP}
N. Garofalo \& A. Petrosyan, \emph{Some new monotonicity formulas and the singular set in the lower-dimensional obstacle problem}, Invent. Math. \textbf{177}~(2009), no.2, 415--461.

\bibitem{GPG}
N. Garofalo, A. Petrosyan \& M. Smit Vega Garcia, \emph{An epiperimetric inequality approach to the regularity of the free boundary in the Signorini problem with variable coefficients}. J. Math. Pures Appl. (9) \textbf{105}~ (2016), no.~6, 745--787.

\bibitem{GPPS}
N. Garofalo, A. Petrosyan, C. A. Pop \& M. Smit Vega Garcia,  \emph{Regularity of the free boundary for the obstacle problem for the fractional Laplacian with drift}, Ann. Inst. H. Poincar\'e Anal. Non Lin\'aire \textbf{34}~(2017), no.~3, 533--570. 

\bibitem{GRO}
N. Garofalo \& X. Ros-Oton, \emph{Structure and regularity of the singular set in the obstacle problem for the fractional Laplacian}. Rev. Mat. Iberoam. \textbf{35}~(2019), no. 5, 1309-1365.

\bibitem{GG}
N. Garofalo \& M. Smit Vega Garcia, \emph{New monotonicity formulas and the optimal regularity in the Signorini problem with variable coefficients}.
Adv. Math. \textbf{262}~ (2014), 682-750.

\bibitem{Gia}
M. Giaquinta, \emph{Multiple integrals in the calculus of variations and nonlinear elliptic systems}. Annals of Mathematics Studies, 105. Princeton University Press, Princeton, NJ, 1983. vii+297 pp. 

\bibitem{JP}
S. Jeon \& A. Petrosyan, \emph{Almost minimizers for certain fractional variational problems}, arXiv:1905.11961

\bibitem{JPS}
S. Jeon, A. Petrosyan \& M. Smit Vega Garcia, \emph{Almost minimizers for the thin obstacle problem with variable coefficients}, arXiv:2007.07349

\bibitem{Ki}
D. Kinderlehrer, \emph{The smoothness of the solution of the boundary obstacle problem}, 
J. Math. Pures Appl. (9) 60 (1981), no. 2, 193-212.

\bibitem{KRS1}
H. Koch, A. R\"uland \& W. Shi, \emph{The variable coefficient thin obstacle problem: Carleman inequalities}. Adv. Math. \textbf{301}~ (2016), 820-866. 

\bibitem{KRS2}
H. Koch, A. R\"uland \& W. Shi, \emph{The variable coefficient thin obstacle problem: optimal regularity and regularity of the regular free boundary}. Ann. Inst. H. Poincar\'e Anal. Non Lin\'eaire \textbf{34}~ (2017), no. 4, 845-897. 

\bibitem{Ne}
A. Nekvinda, \emph{Characterization of traces of the weighted Sobolev space $W^{1,p}(\Om,d^{\ve}M)$ on $M$}, Czechoslovak Math. J., \textbf{43}~ (1993), 695-711.

\bibitem{PSU}
A. Petrosyan, H. Shahgholian \& N. Uralsteva, \emph{Regularity of free boundaries in obstacle-type problems},
Graduate Studies in Mathematics, 136. American Mathematical Society, Providence, RI, 2012. x+221 pp.

\bibitem{STV1}
Y. Sire, S. Terracini \& S. Vita, \emph{Liouville type theorems and regularity of solutions to degenerate or singular problems part I: even solutions}, arXiv:1904.02143

\bibitem{STV2}
Y. Sire, S. Terracini \& S. Vita, \emph{Liouville type theorems and regularity of solutions to degenerate or singular problems part II: odd solutions}, arXiv:2003.09023 



\bibitem{TX}
J. Tan \& J. Xiong, \emph{A Harnack inequality for fractional Laplace equations with lower order terms}, Discrete Contin. Dyn. Syst. \textbf{31}~(2011), no. 3, 975-983.

\end{thebibliography}
\end{document}